\documentclass[12pt,reqno]{amsart}
\usepackage{bm}
\usepackage{graphicx}
\usepackage{adjustbox}
\usepackage{amsthm,amssymb,amsmath}
\usepackage[T1]{fontenc}
\usepackage[utf8]{inputenc}
\usepackage[colorlinks=true, linkcolor=blue, urlcolor=blue, citecolor=blue, anchorcolor=blue]{hyperref}
\usepackage{xcolor}
\usepackage{mathrsfs}
\usepackage{enumitem}
\usepackage{setspace}
\usepackage{yfonts}
\usepackage{float}
\usepackage{mathtools}
\usepackage[all,cmtip]{xy}
\usepackage{todonotes}
\usepackage{tikz-cd}
\usepackage{relsize}
\newcommand{\p}{\mathbb{P}_{C}}
\newtheorem{manualtheoreminner}{Theorem}
\newenvironment{manualtheorem}[1]{%
\renewcommand\themanualtheoreminner{#1}%
  \manualtheoreminner
}{\endmanualtheoreminner}

\newtheorem{manualquestioninner}{Question}
\newenvironment{manualquestion}[1]{%
  \renewcommand\themanualquestioninner{#1}%
  \manualquestioninner
}{\endmanualquestioninner}

\newtheorem{manualprobleminner}{Problem}
\newenvironment{manualproblem}[1]{%
  \renewcommand\themanualprobleminner{#1}%
  \manualprobleminner
}{\endmanualprobleminner}

\newtheorem{manualcorollaryinner}{Corollary}
\newenvironment{manualcorollary}[1]{%
  \renewcommand\themanualcorollaryinner{#1}%
  \manualcorollaryinner
}{\endmanualcorollaryinner}

\usepackage[backend=bibtex,style=alphabetic]{biblatex}

\tikzcdset{row sep/normal=50pt, column sep/normal=50pt}

\newtheorem{lemma}{Lemma}[section]
\newtheorem{theorem}[lemma]{Theorem}

\newtheorem{corollary}[lemma]{Corollary}

\theoremstyle{definition}

\newtheorem{remark}[lemma]{Remark}

\numberwithin{equation}{section}

\begin{document}

\title[Automorphisms of Hilbert schemes and symmetric powers]{Automorphisms of  punctual Hilbert schemes and symmetric powers of varieties}

\author[A. Bansal]{Ashima Bansal}

\address{Indian Institute of Science Education and Research Tirupati, Srinivasapuram, Yerpedu Mandal, Tirupati Dist, Andhra Pradesh, India – 517619.}

\email{ashimabansal@students.iisertirupati.ac.in}

\author[S. Sarkar]{Supravat Sarkar}

\address{\scshape Fine Hall, Princeton, NJ 08544}

\email{ss6663@princeton.edu}

\author[S. Vats]{Shivam Vats}

\address{Indian Institute of Science Education and Research Tirupati, Srinivasapuram, Yerpedu Mandal, Tirupati Dist, Andhra Pradesh, India – 517619}

\email{shivamvatsaaa@gmail.com}

\subjclass[2010]{ 14J50}

\keywords{Punctual Hilbert scheme, Symmetric power,  Automorphism}

\begin{abstract}
We classify complex smooth projective surfaces whose punctual Hilbert scheme has a non-natural automorphism preserving the big diagonal. This completely answers a question raised by Belmans, Oberdieck and Rennemo, and extends previous works by Boissi{\`e}re-Sarti, Hayashi, Sasaki, Girardet and Wang. We reduce this to studying the existence of non-natural automorphisms of symmetric powers. We study this question for higher dimensional varieties too, giving some sufficient conditions guaranteeing every automorphism of a symmetric power to be natural. As a corollary, we characterize smooth projective surfaces of Kodaira dimension $\geq 1$ whose punctual Hilbert scheme has a non-natural automorphism, this time not assuming the automorphism preserves the big diagonal. We also address the question when a smooth projective variety is determined up to isomorphism by its punctual Hilbert scheme.
\end{abstract}
\maketitle

\section{Introduction}
\noindent 
We work throughout, over the field $\mathbb{C}$ of complex numbers. For a smooth projective variety $X$ of dimension $n$, the punctual Hilbert scheme Hilb$^m(X)$ parametrizing the length $m$ subschemes of $X$ is a widely studied object. \cite{fogarty1968algebraic} showed Hilb$^m(X)$ is always smooth irreducible if $n\,=\,2$. For $n>2$ and $m>3$, Hilb$^m(X)$ is always singular by \cite{skjelnes2020smooth}. For a K3 surface $X$, Hilb$^m(X)$ gives an example of the hyperkahler variety.

For a smooth projective surface $X$, the study of automorphisms of Hilb$^m(X)$ is a fascinating area of research. Any automorphism of $X$ induces an automorphism of Hilb$^m(X)$, these are called \textit{natural} automorphisms. There might be non-natural automorphisms of Hilb$^m(X)$: for $X$ a general quartic K3 surface, the automorphism of Hilb$^2(X)$ sending a subscheme to the residual subscheme of the line passing through it is a non-natural automorphism, see {\cite[Section 6]{beauville1983some}}. One sees that any natural automorphism $\phi$ of Hilb$^m(X)$ preserves the big diagonal $E$ parametrizing non-reduced subschemes of $X$, in the sense that $\phi(E)=E.$ The above automorphism does not preserve $E$, so one can ask whether any automorphism of $\text{Hilb}^m(X)$ preserving $E$ is natural. The answer is again no, and it is easy to give a non-natural automorphism of Hilb$^2(C_1\,\times\, C_2)$, where $C_1$ and $C_2$ are curves, see \cite{belmans2020automorphisms}. Based on this, \cite{belmans2020automorphisms} asks the following question:

\begin{manualquestion}{1}\label{q}
Suppose $X$ is a smooth projective surface and $\phi$ an automorphism of $\emph{Hilb}^m(X)$ preserving the big diagonal. Excluding the case $X\,=\,C_1\,\times\, C_2$ and $m\,=\,2$, does it follow that $\phi$ is natural?
\end{manualquestion}

 \cite{belmans2020automorphisms} answers Question \ref{q} affirmatively for $X$ weak Fano or general type. By \cite{boissiere2012note} and {\cite[Theorem 1.2]{hayashi2015universal}}, Question \ref{q} has an affirmative answer when $X$ is a K3 surface or an Enriques surface. Using {\cite[Theorems 2.6, 2.7]{Ha}} and {\cite[Proposition 9]{belmans2020automorphisms}}, one sees that Question \ref{q} has an affirmative answer for surfaces with irregularity $0$. \cite{Gir} shows Question \ref{q} has positive answer for abelian surfaces of Picard rank $1$. On the other hand, \cite{Sas} and \cite{Gir} give a negative answer to Question \ref{q} for some abelian surfaces $X.$ So, it makes sense to consider the following problem:

\begin{manualproblem}{1}\label{p}
Classify the pairs $(m,X)$ of smooth projective surfaces $X$ and integers $m\,\geq\, 2$ such that $\emph{Hilb}^m(X)$ has a non-natural automorphism preserving the big diagonal.   
 \end{manualproblem}

One of the goals of this paper is to completely solve Problem \ref{p}. Even when dim $X\geq 3$, we want to give sufficient conditions for any automorphism of $\text{Hilb}^m(X)$ preserving the big diagonal to be natural. We assume $\text{Hilb}^m(X)$ is smooth. For dim $X\,\geq\, 3$, this is equivalent to saying $m\leq 3$ by \cite{skjelnes2020smooth}. Note that the big diagonal is the exceptional divisor of the Hilbert-Chow morphism $\text{Hilb}^m(X)\,\to\, S^mX,$ where $S^m X$ is the $m$'th symmetric power of $X$. To state our result, we need the following notation: let $\mathcal{C}$ be the class of smooth projective surfaces $X$ which are of the form $(D\,\times\, Y)/G$, where $D$ is an elliptic curve, $Y$ a smooth curve of general type, and $G$ is a finite subgroup of $\text{Aut}^0(D)$ acting diagonally on $D\times Y.$ This is what was called class $E$ in {\cite[Theorem D]{Fon}}. In this case, $X$ is minimal of Kodaira dimension $1$, and the Iitaka fibration of $X$ is an isotrivial elliptic fibration.

Our result is the following, generalizing the previous works stated above:

\begin{manualtheorem}{A}\label{1}
    \begin{enumerate}
        \item Let $X$ be a smooth projective surface and $m\,\geq\, 2$ be an integer. Then $\emph{Hilb}^m(X)$ has a non-natural automorphism preserving the big diagonal if and only if one of the following holds:
        \begin{enumerate}
            \item $m\,=\,2$ and $X$ is a product of curves,
            \item $X$ is an abelian surface isogenous to the square of an elliptic curve,
            \item $X$ is a simple abelian surface and $\emph{End}_{\mathbb{Q}}(X)\,:=\,\emph{End}(X)\,\otimes\, \mathbb{Q}$ is not $\mathbb{Q}$ or an imaginary quadratic extension of $\mathbb{Q}$,
            \item $X$ is in class $\mathcal{C}$, and if $C$ and $E$ are the base and the fibre of the Iitaka fibration of $X$, then there is a non-constant map $C\,\to\, E.$
        \end{enumerate}
        \item Let $X$ be a smooth projective variety of dimension $n\,\geq\, 3$, and let $m\,=\,2$ or $3$. Suppose either $H^{1}(X, \mathcal{O}_{X})\, =\, 0$ or ${K}_{X}$ is ample. Also, assume $X$ is indecomposable if $m\,=\,2$. Then any automorphism of $\emph{Hilb}^{m}(X)$ preserving the big diagonal is natural.
    \end{enumerate}
\end{manualtheorem}

The negative answers to Question \ref{q} due to $(b)$ and $(c)$ of Theorem \ref{1} are of a similar nature as \cite{Sas}, $(d)$ gives a new class of negative answers to Question \ref{q}, and our theorem says that these are all the negative answers.

As a corollary, for surfaces $X$ of Kodaira dimension $\geq\, 1$, we can characterize when $\textrm{Hilb}^{m}(X)$ has a non-natural automorphism. Here, we are not assuming that the automorphism preserves the big diagonal.

\begin{manualcorollary}{A}\label{c1}
    Let $X$ be a smooth projective surface of Kodaira dimension $\kappa(X)\,\geq\, 1$, and let $m\,\geq\, 2$ be a positive integer. Then $\textrm{Hilb}^{m}(X)$ has a non-natural automorphism if and only if $(a)$ or $(d)$ in Theorem \ref{1} holds.
\end{manualcorollary}

Theorem \ref{1} will be an immediate consequence of Theorems \ref{2}, \ref{3} and \ref{4}. In Theorem \ref{2}, we reduce Theorem \ref{1} to the analogous question for automorphism of symmetric power. Recall that for a variety $X$, and a positive integer $m$, the symmetric power $S^mX$ is the quotient $X^m/S_m,$ where the symmetric group $S_m$ on $m$ letters acts on $X^m$ by permuting the factors. 

\begin{manualtheorem}{B}\label{2}
    Let $m,n\,\geq\,2$ be integers, and either $n\,=\,2$ or $m\,\leq\, 3.$ Let $X$ be a smooth projective variety of dimension $n$. Then $\emph{Hilb}^{m}X$ has a non-natural automorphism preserving the big diagonal if and only if $S^{m}X$ has a non-natural automorphism.
\end{manualtheorem}

Now we investigate when the symmetric power of a smooth projective variety $X$ of dimension $n$ can have a non-natural automorphism. This seems to be the right analogue of Problem \ref{p} for higher dimensional $X$, at least when the Hilbert scheme is not irreducible. For $n\,=\,1$, this has been studied in \cite{ciliberto1993symmetric}, \cite{ciliberto1995singularities}, \cite{martens1965extended}, \cite{weil2009beweis} and \cite{biswas2017automorphisms}. \cite{belmans2020automorphisms}, \cite{Ha} and \cite{Og} studied this for $n\,=\,2$,  \cite{belmans2020automorphisms} for $X$ projective space and $m\,=\,2$, \cite{Wa} for $X$ a smooth projective hypersurface, $n\geq 3$ and $m\,=\,2$, {\cite[Theorem 3.1.1]{Sh}} for $X$ a simply connected variety of Picard rank $1$. For a variety $Z$, let us denote the group of automorphisms of $Z$ by Aut$(Z)$. Without any restriction on dimension $n\,\geq\, 2$, our following result gives a sufficient criterion for every automorphism of $S^m X$ to be natural, generalizing {\cite[Theorem 2.7]{Ha}}.

\begin{manualtheorem}{C}\label{3}
 Let $X$ be a smooth projective variety of dimension $n\, \geq\, 2$, and $m\,\geq\,2$ is an integer. Then we have:
 
\begin{enumerate}
\item[(i)] Suppose either $H^{1}(X, \mathcal{O}_{X})\, =\,0$, or $K_X$ is ample. If $m\,\geq\, 3$, every automorphism of $S^{m}(X)$ is natural. For $m\,=\,2$, we have a short exact sequence of groups $$1\,\longrightarrow\, (\mathbb{Z}/2\mathbb{Z})^{k-1}\,\longrightarrow\, \textnormal{ Aut} (S^2X)\,\longrightarrow\, \textnormal{Aut} (X)\,\longrightarrow \,1,$$ and the inclusion of $\textnormal{Aut} (X)$ in $\textnormal{ Aut} (S^2X)$ as the subgroup of natural automorphisms induces a splitting of this short exact sequence. Here $k$ is the number of factors in a product decomposition of $X$ into indecomposable varieties.
\item [(ii)] Let $m\,\geq\,2$ be an integer, $Y$ a smooth projective variety. If dim $Y\,\geq\, 2$, assume every automorphism of $S^mY$ is natural. Let $X$ be a $\mathbb{P}^r$-bundle over $Y$. Assume that every map $\mathbb{P}^r\,\to\, Y$ is constant, and if $m\,=\,2,$ $X\,\not\,\cong\, \mathbb{P}^r\,\times\, Y $ over $Y$. Then every automorphism of $S^m X$ is natural.
\end{enumerate}
\end{manualtheorem}

Finally, we restrict our attention when $X$ is a surface. We completely describe the pairs $(m,X)$ for which $S^mX$ has a non-natural automorphism.

\begin{manualtheorem}{D}\label{4}
Let $X$ be a smooth projective surface and $m\,\geq\, 2$ be an integer. Then $S^m(X)$ has a non-natural automorphism if and only if one of the following holds:
\begin{enumerate}
\item $m\,=\,2$ and $X$ is a product of curves,
\item $X$ is an abelian surface isogenous to the square of an elliptic curve,
\item $X$ is a simple abelian surface and $\emph{End}_{\mathbb{Q}}(X)\,:=\,\emph{End}(X)\,\otimes\, \mathbb{Q}$ is not $\mathbb{Q}$ or an imaginary quadratic extension of $\mathbb{Q}$,
\item $X$ is in class $\mathcal{C}$, and if $C$ and $E$ are the base and the fibre of the Iitaka fibration of $X$, then there is a non-constant map $C\,\to\, E.$
\end{enumerate}
\end{manualtheorem}

We prove Theorem \ref{2} in \S 3. There we also prove Lemma \ref{auto of Pbundle} classifying some projective bundles which has an automorphism not over an automorphism of the base, which is of independent interest. We prove Theorems \ref{3} $(i)$, \ref{4} and Corollary \ref{c1} in \S 4. Finally in \S 5 we prove Theorem \ref{4} and \ref {3} $(ii)$. Theorem \ref{1} is an immediate consequence of Theorems \ref{2}, \ref{3}  $(i)$ and \ref{4}. We also address the question when a smooth projective variety is determined up to isomorphism by its punctual Hilbert scheme in Remark \ref{determine} and \ref{determine1}, extending {\cite[Corollary 5]{belmans2020automorphisms}} and {\cite[Theorem 1.4]{Ha}}.
\section{Notations and conventions}

\begin{itemize}
\item A \textit{variety} is an integral separated scheme of finite type over $\mathbb{C}$. Unless otherwise stated, any point of a variety will mean a closed point.
\item A projective variety $Z$ is called \textit{indecomposable} if $Z$ cannot be written as $Z_{1}\,\times\, Z_{2}$ where neither $Z_{1}$ nor $Z_{2}$ is a point.
\item A proper morphism $f\,:\,Y\,\to\, X$ between varieties is said to be a \textit{contraction} if $f_* \mathcal{O}_Y\,=\,\mathcal{O}_X$. We shall occasionally use the following standard fact: if $f\,:\,Y\,\to\, X$ is a contraction of normal varieties and $g\,:\,Y\,\to\, W$ is a morphism of normal varieties such that every fibre of $f$ is mapped to a point by $g$, then $g$ factors uniquely through $f.$
    \item If $X$ is a smooth projective surface and $C$ a smooth projective curve, a contraction $f\,:\, X\,\to\, C$ is called \textit{isotrivial} if there is a dense open set $U$ in $C$ such that all fibres of $f$ over points of $U$ are isomorphic. We call $f$ \textit{non-isotrivial} if it is not isotrivial.
    \item The identity automorphism for a variety will be denoted by $id$.
    \item Let $f\,:\,X\,\to\, Y$ be a surjective  proper morphism of normal varieties. Given an automorphism $\phi$ of $X,$ we say $\phi$ \textit{descends} to an automorphism $\psi$ of $Y$ if the following diagram commutes.

\begin{equation}\label{descend diagram}
\begin{tikzcd}
X \arrow[r, "\phi"] \arrow[d,"f"]
& X \arrow[d, "f" ] \\
Y  \arrow[r,"\psi" ]
& |[, rotate=0]| Y
\end{tikzcd}
\end{equation}

In this case we say $\phi$ is \textit{over an automorphism of} $Y$, and \textit{over} $Y$ if $\psi\,=\,id.$
In many places in this article, we would want to descend an automorphism $\phi$ of $X$ to $Y$. The way we will do it is the following, which we will not elaborate later:  we will first descend $\phi$ to a morphism $\psi\,:\,Y\,\to\,Y$ making diagram \ref{descend diagram} commute. Then the same argument will show that $\phi^{-1}$ descends to a morphism $\psi_1\,:\,Y\,\to\, Y$. Now surjectivity of $f$ implies $\psi$ and $\psi_1$ are inverses to each other, so both are automorphisms of $Y$.

   In fact, one can directly show that if $\phi$ descends to a morphism $\psi\,:\,Y\,\to\, Y$, then $\psi$ is necessarily an automorphism of $Y$, but we will not need this fact.
  \item  Let $f\,:\,X\,\to\, Y$ be a surjective  proper morphism of normal varieties. Given an automorphism $\psi$ of $Y,$ we say $\psi$ \textit{lifts} to an automorphism $\phi$ of $X$ if diagram \ref{descend diagram} commutes.

  \item For a variety $X$, Sing $X$ denotes the singular locus of $X.$
  \item For a point $x$ in a variety $X$, $T_x X$ denotes the Zariski tangent space of $X$ at $x.$ For a smooth variety $X$, the tangent and cotangent bundles are denoted by $T_X$ and $\Omega_X$, respectively.

  \item For a group homomorphism $ G_{1}\, \xlongrightarrow{f} \,G_{2},$ we denote the image of $f$ by $ \textnormal{im}(f).$ The center of a group $G$ is denoted by $ \mathcal{Z}(G).$ 
  \item For a point $p$ on a complex analytic set $X$ we denote the germ of $X$ at $p$ by $\langle X, p\rangle$.
  \item Recall the following notation from \cite{BSV3}. For a smooth variety $Y$ of dimension $n\,\geq\, 2$, and $ \pi\,=\, (a_{1},\cdots,a_{k})$ a partition of a positive integer $m$. Let 
  $$ W_{\pi,m}(Y)\,=\, \{ \sum_{i=1}^{k}a_{i}y_{i}\, \in\, S^{m}Y \, \, | \, \, y_{1},y_{2},\cdots,y_{k} \in Y\}$$ and
  $$ W_{\pi,m}^{\circ}(Y)\,=\, \{ \sum_{i=1}^{k}a_{i}y_{i}\, \in\, S^{m}Y \, \, | \, \, y_{1},y_{2},\cdots,y_{k} \in Y \text{ are distinct}\}.$$
We would write them as $W_{\pi}$ and $W_{\pi}^{\circ}$ when $m$ and $Y$ are clear from the context.  For $ \pi\,=\, (2, \underline{1}),$ we would abbreviate $W_{\pi}$ and $W_{\pi}^{\circ}$ as $W$ and $W^{\circ}$, respectively. Also, for a partition $ \pi\,=\, (a_{1},\cdots,a_{k})$ of a positive integer $m$, we define $|\pi|\,:=\,k.$
  \item \label{notation} For a smooth variety $Y$ of dimension $n\geq 2$, and the Hilbert-Chow morphism $\textrm{Hilb}^{m}(Y)\,\xrightarrow{\,\,\,p\,\,\,}\, S^{m} Y$, define $E_{{\pi},m}(Y)\,=\, p^{-1}(W_{\pi})$, $E^{\circ}_{{\pi},m}(Y)\,=\, p^{-1}(W^{\circ}_{\pi})$, under reduced induced structures. We would write them as $E_{\pi}$ and $E_{\pi}^{\circ}$ when $m$ and $Y$ are clear from the context. By {\cite[Lemma 3.6.1]{DMM}}, we have $E_{\pi}=\overline{E_{\pi}^{\circ}}$ for all $\pi$. For $ \pi\,=\, (2, \underline{1}),$ we would abbreviate $E_{\pi}$ and $E_{\pi}^{\circ}$ as $E$ and $E^{\circ}$, respectively. 
   \item For a smooth projective variety $X$ An automorphism of $\textnormal{Hilb}^{m}(X), S^mX$ or $X^m$ is called \textit{natural} if it is induced by an automorphism of $X$, and called \textit{non-natural} if it is not natural.
  \item For a smooth projective variety $X$, Aut $X$ denotes the automorphism group scheme of $X$. For a morphism $X\to Y$ of smooth projective varieties, Aut$_Y X$ denotes the closed subgroup scheme of Aut $X$ consisting of automorphisms over $Y$. Aut$^{\circ} X$ and Aut$^{\circ}_Y X$ denote the connected component of identity of these group schemes.
  \item For a normal projective variety $X$, $\rho(X)$ denotes the Picard rank of $X$. For a contraction $X\,\to\, Y$ of normal projective varieties, the relative Picard rank of $X$ over $Y$ will be denoted by $\rho(X/Y).$
  \item  We follow the convention of projective bundle as in {\cite[Chapter 2, Section 7]{Hart}}. For a smooth variety $X$, a $\mathbb{P}^r$--bundle over $X$ means the projectivization of a vector bundle of rank $r\,+\,1$ on $X$. A $\mathbb{P}$--bundle means a $\mathbb{P}^r$-bundle for some $r.$ Two $\mathbb{P}$--bundle structures on $Y$ given by $f_i\,:\,Y\,\to\, X_i$ for $i\,=\,1,2$ are considered same if there is an isomorphism $\psi\,:\, X_1\,\to\, X_2$ such that $f_2\,=\,\psi\circ f_1.$
  \item For a variety $X$ and a positive integer $m$, we denote the diagonal in $X^m$ by $\Delta _X.$
  \item For a morphism of varieties $f\,:\,X\,\to\, Y$, the induced morphism $X^m\,\to\, Y^m$ will be denoted by $f_{(m)}.$
 
  \item The genus of a smooth projective curve $C$ will be denoted by $g(C).$ For a smooth projective variety $X$, $a(X)$ will denote the Albanese dimension of $X$, that is, the dimension of the image of the Albanese map $X\,\to\, \textrm{Alb}(X)$.
  \item For a variety $X$, $\underline{x}\,=\,(x_i)_i\,\in\, X^m$ and $\sigma\,\in\, S_m$, $\underline{x}_{\sigma}$ will denote the element $(x_{\sigma(i)})_i$ of $X^m.$
  \item For a ring $R$, its group of units will be denoted by $R^{\times}$. We shall use the following standard fact: if $K$ is a finite dimensional $\mathbb{Q}$--algebra, and one order in $K$ has infinitely many units, then each order in $K$ has infinitely many units.
  \item For an abelian variety $X$, $\textrm{End}(X)$ will denote the endomorphism ring of $X$, which is an order in the finite dimensional $\mathbb{Q}$-algebra $\textrm{End}_{\mathbb{Q}}(X)\,:=\,\textrm{End}(X)\,\otimes\, \mathbb{Q}$. The above fact implies that if $X$ and $Y$ are isogenous abelian varieties and $\textrm{End}(X)^{\times}$ is infinite, then $\textrm{End}(Y)^{\times}$ is infinite too.
\end{itemize}

\section {Descent and lift of automorphisms}
In this section, we prove Theorem \ref{2}. It follows immediately from the following two theorems.
\begin{theorem}\label{descent}
Under the assumptions of Theorem \ref{2}, any automorphism of $\emph{Hilb}^m(X)$ that preserves the big diagonal descends to an automorphism of $S^m X$ under the Hilbert--Chow morphism. 
\end{theorem}

\begin{theorem}\label{lift}
    
Under the assumptions of Theorem \ref{2}, any automorphism of $S^{m}X$ lifts to an automorphism of $\emph{Hilb}^{m}X$ that preserves the big diagonal. 
\end{theorem}

We now need some lemmas.

 \begin{lemma}\label{Pbundle}
Let $X$ be a smooth projective variety of dimension $n$, endowed with two different $\mathbb{P}$-bundle structures 
$$ X \,\xlongrightarrow{\phi}\,Y, \quad  X\, \xlongrightarrow{\psi}\, Z.$$ 
Then the following holds:

$(i)$We have 
$$\textnormal{dim}\,\, Y\, +\, \textnormal{dim}\, \, Z \,\geq\, n,$$ 
with equality holds if and only if $Y$ and $Z$ are projective spaces and $X$ their product.

$(ii)$ If
$$ \textnormal{dim}\,\,Y\,+\, \textnormal{dim}\,\,Z\,=\, n\,+\,1, $$
then either:
$n\,=\,2m\,-\,1$, $Y\,=Z\,=\, \mathbb{P}^m$ and $X\,=\,\mathbb{P}(T_{\mathbb{P}^m})$; or $ Y$ and $Z$ each admit a $ \mathbb{P}$-bundle structure over a smooth curve $C$, and $ X\,=\, Y\, \times_{C}\, Z$.
\end{lemma}
\begin{proof}
    Follows from \cite[Theorem 2]{occhetta2002euler} and the remark before that.
\end{proof}

We will only use Corollary \ref{Pomega} of the following Lemma, but we include this much stronger Lemma as it is interesting on its own. 
     
\begin{lemma}\label{auto of Pbundle}
       Let $X$ be a smooth projective variety of dimension $n\, \geq\, 2$, $W$ a $ \mathbb{P}^r$--bundle over $X$, where $ r\, \geq\, n\,-\,1$ is an integer. Then $W$ has an automorphism $\phi$  which is not over an automorphism of $X$ if and only if either
       
       $1)$ $ X\,=\, \mathbb{P}^n$, $ W\,=\, \mathbb{P}^n \,\times\, \mathbb{P}^n$, or

       $2)$ $ X\,=\, \mathbb{P}^n$, $ W\,=\, \mathbb{P}_{\mathbb{P}^n}(T_{\mathbb{P}^n})$, or

       $3)$ There is a smooth curve $C$, a rank $n$ vector bundle $ \mathcal{E}$ over $ C$, and $ \psi\, \in\, \textnormal{Aut} \, \, C$ with $ (\psi^2)^{\ast} \mathcal{E}\, \cong\, \mathcal{E}$ up to line bundle twist such that  $ W\,=\, \mathbb{P}_{C}(\mathcal{E})\, \times_{C}\, \mathbb{P}_{C}(\psi^{\ast}\mathcal{E})$, $ X\,= \,\mathbb{P}_{C}(\mathcal{E}),$ and $ W \,\longrightarrow\, X$ is the projection onto first factor.
        \end{lemma}
\begin{proof}
$(\Longleftarrow)$ If $(1)$ or $(2)$ holds, the statement is clear. 

If $(3)$ holds, note the following sequence of isomorphisms:
 \begin{equation*}
\p(\mathcal{E})\,\times_{C}\,\p(\psi^{\ast}\mathcal{E})\,\cong\, \p((\psi^{2})^{\ast}\mathcal{E})\,\times_{C}\, \p(\psi^{\ast}\mathcal{E})\,\cong\, \psi^{\ast}(\p(\psi^{\ast}\mathcal{E})\,\times_{C}\,\p(\mathcal{E}))
\end{equation*}
\begin{equation*}
\hspace{2.4cm} \cong\psi^{\ast}(\p(\mathcal{E})\times_{C}\p(\psi^{\ast}\mathcal{E}))\cong \p(\mathcal{E})\times_{C} \p(\psi^{\ast}\mathcal{E}),
\end{equation*}     
whose composition gives an automorphism of $W$ that is not induced by an automorphism of $X$.

$ (\Longrightarrow) $ Let $ W\, \xlongrightarrow{p}\, X$ be the projection. Since $ \phi$ is not over an automorphism of $X$, the variety $W$ admits two different $ \mathbb{P}^r$-bundle structure over $X$\,:\, $p$ and $ p\, \circ\, \phi $.

If $ r\,\geq\, n$, then the conclusion of $(1)$ holds by Lemma \ref{Pbundle} $(i)$. So assume $r\,=\, n\,-\,1$.  By Lemma \ref{Pbundle} $(ii)$, two cases can occur:

     $ \underline{\textrm{Case} \,1}:$  $ X\,=\, \mathbb{P}^n$,  $ W\,=\, \mathbb{P}(T_{\mathbb{P}^n})$. 

     In this case, the conclusion $(2)$ holds.

     $ \underline{\textrm{Case} \,\, 2}:$ There exists a smooth curve $C$ morphisms 
     $$ X\, \xlongrightarrow{f}\, C, \quad X \,\xlongrightarrow{g}\,C$$ 
     
     which defines $ \mathbb{P}^{n-1}$-bundle structures such that the diagram 

\begin{equation}\label{pbundle diagram1}
\begin{tikzcd}
W \arrow[r, "p \circ \phi"] \arrow[d,"p"]
& X \arrow[d, "g" ] \\
X  \arrow[r,"f" ]
& |[, rotate=0]| C
\end{tikzcd}
\end{equation}

is commutative and Cartesian. 

If $f$ and $g$ are different $ \mathbb{P}^{n-1}$--bundle structures on $X$, then by Lemma \ref{Pbundle} $(i)$, we get $ X\, =\, \mathbb{P}^1 \,\times\, \mathbb{P}^1$, with $f$ and $g$ being the two projections. In this case,

$$ W\,=\,(\mathbb{P}^{1})^{3}\,=\, \mathbb{P}_{\mathbb{P}^1}(\mathcal{O}_{\mathbb{P}^1}^{2})\,\times_{\mathbb{P}^1}\, \mathbb{P}_{\mathbb{P}^1}(\mathcal{O}_{\mathbb{P}^1}^{2}),$$ 
so the conclusion of $(3)$ holds. 

Now, suppose $f$, $g$ are the same $ \mathbb{P}^{n-1}$--bundle structure. Then there exists $ \psi\, \in\, \textnormal{Aut}(C)$ and a rank $n$ vector bundle $ \mathcal{E}$ on $C$ such that $ X\, \xlongrightarrow{f}\,C$ is the projection $ \mathbb{P}_{C}(\mathcal{E})\, \longrightarrow\, C$, and $ g\,=\, \psi\, \circ\, f.$ 

Thus, the morphism $ X\, \xlongrightarrow{g}\, C$ can  be identified with the projection 
$$\mathbb{P}_{C}((\psi^{-1})^{\ast}\mathcal{E})\,\longrightarrow\, C.$$ 

Now diagram \ref{pbundle diagram1} shows that we have the following commutative diagram 

\[\begin{tikzcd}
	{\mathbb{P}_{X}(g^\star\mathcal{E})} & W & W & \mathbb{P}_{X}(f^{\ast}(\psi^{-1})^{\ast}\mathcal{E}) \\
	& X
	\arrow[from=1-1, to=2-2]
	\arrow[phantom,"\cong", no head, from=1-1, to=1-2]
	\arrow[phantom,"\overset{\phi}{\cong}", no head, from=1-2, to=1-3]
	\arrow[phantom,"\cong", no head, from=1-3, to=1-4]
	\arrow[ "p \circ \phi",from=1-2, to=2-2]
	\arrow["p", from=1-3, to=2-2]
	\arrow[from=1-4, to=2-2]
\end{tikzcd},\]

where the leftmost and rightmost maps are natural projections. Hence, 
$$ \mathbb{P}_{X}(f^{\ast}(\psi^{-1})^{\ast}\mathcal{E})\, \cong\, \mathbb{P}_{X}(g^{\ast}\mathcal{E})\, \cong\, \mathbb{P}_{X}(f^{\ast} \psi^{\ast}\mathcal{E})$$ 
over $X$. It follows that there exists a line bundle $L$ on $X$ such that 
$$ f^{\ast}(\psi^{-1})^{\ast}\mathcal{E}\, \cong\, f^{\ast} \psi^{\ast}\mathcal{E} \,\otimes\, L.$$ 
Using the same argument as in \cite[ Lemma 3.1]{bansal2023isomorphism}, it shows that 
$$ \psi^{\ast}\mathcal{E}\, \cong\, (\psi^{-1})^{\ast} \mathcal{E}\,\otimes\, M$$ 
for some line bundle $M$ on $C$. More precisely, applying $f_*$ and the projection formula, we obtain  $$(\psi^{-1})^{\ast}\mathcal{E} \,\cong\, \psi^{\ast}\mathcal{E} \,\otimes\, f_{\ast}L.$$ Since $\psi^{\ast}\mathcal{E}$ and $(\psi^{-1})^{\ast}\mathcal{E}$ are locally trivial and of the same rank, it follows that $f_{\ast} L$ is a line bundle.  Therefore, $ \mathcal{E}\,\cong\, (\psi^{2})^{\ast}\mathcal{E}$ up to a line bundle twist, and hence the conclusion of $(3)$ holds.
\end{proof}
\begin{corollary}\label{Pomega}
Let $X$ be a smooth projective variety of dimension $\geq\, 2$. Then any automorphism of $\mathbb{P}_{X}(\Omega_{X})$ is over an automorphism of $X$, unless $X\, \cong\, \mathbb{P}^{2}$.
\end{corollary}
\begin{proof}
Suppose $X \,\not\cong\, \mathbb{P}^{2}$. Let $W \,=\, \mathbb{P}_{X}(\Omega_{X})$, and let
$$W\,\xrightarrow{\,\,\, p\,\,\,}\, X$$ be the natural projection. Suppose $W$ admits an automorphism that is not over an automorphism of $X$. Then, one of the cases $(1)$, $(2)$ or $(3)$ in Lemma $\ref{auto of Pbundle}$ must occur.

An examination of Chern classes shows that $\Omega_{\mathbb{P}^{n}}$ is not isomorphic to $\mathcal{O}_{\mathbb{P}^{n}}^{n}$ or to $T_{\mathbb{P}^{n}}$, up to line bundle twist, if $n\,>\,2$. Therefore, cases $(1)$ or $(2)$ of Lemma \ref{auto of Pbundle} cannot occur. Thus, case $(3)$ of Lemma \ref{auto of Pbundle} must occur. Hence $X \,=\, \p(\mathcal{E}_{1})$ for some rank $n$ vector bundle $\mathcal{E}_{1}$ over a curve $C$, and $\Omega_{X} \,\cong\, f^{\ast}\mathcal{E}_{2}$ up to line bundle twist, where $\mathcal{E}_{2}$ is another rank $n$ vector bundle on $C$, and $f\,:\,\p(\mathcal{E}_{1})\,\longrightarrow\, C$ is the projection. Now, restrict this isomorphism to a fibre of $f$, which is an isomorphism to $\mathbb{P}^{n-1}$. We obtain
 \begin{equation*}
\Omega_{\mathbb{P}^{n-1}}\,\oplus\,\mathcal{O}_{\mathbb{P}^{n-1}}\,\cong\, \mathcal{O}_{\mathbb{P}^{n-1}}(r)^{\oplus n}
\textrm{for some}\,\, r\,\in\, \mathbb{Z}.
\end{equation*}
Taking $h^{0}$, we get:
\begin{equation*}
1\, =\, h^{0}(\mathbb{P}^{n-1}, \Omega_{\mathbb{P}^{n-1}}\,\oplus\, \mathcal{O}_{\mathbb{P}^{n-1}}) \, =\, h^{0}(\mathbb{P}^{n-1}, \mathcal{O}_{\mathbb{P}^{n-1}}(r)^{\oplus n})   
\end{equation*}
\begin{equation*}
\hspace{5.6cm} =\, n \,h^{0}(\mathbb{P}^{n-1}, \mathcal{O}_{\mathbb{P}^{n-1}}(r)),  
\end{equation*}
which contradicts the assumptions $n\,\geq\, 2$.
\end{proof}
\begin{remark}
The same argument as in Corollary \ref{Pomega} shows that if $X$ is a smooth projective variety of dimension $\geq\, 2$ that is not a projective space, then any automorphism of $\mathbb{P}_{X}(T_{X})$ is over an automorphism of $X$.
\end{remark}

\begin{lemma}\label{Hilbert drum}
Let $n\,\geq\,2$ be an integer. Let $H_{3,n}$ denote the reduced fibre of the Hilbert-Chow morphism $\emph{Hilb}^{3}\,\mathbb{A}^{n}\,\longrightarrow\, S^{3}\mathbb{A}^{n}$ over $3\cdot[\underline{0}]$.  Then $H_{3,n}$ is irreducible and rational of dimension $2(n\,-\,1)$, and its normalization has Picard rank $1$.
\end{lemma}
\begin{proof}
Let $H^{\circ}\,\subset\, H_{3,n}$ be the open subset consisting of subschemes $Z$ supported on  a curve. By \cite[Lemma 3.2]{JT}, $H^{\circ}$ is the total space of a vector bundle of rank $n\,-\,1$ over $\mathbb{P}^{n-1}$. Thus, $H^{\circ}$ is rational and has dimension $2(n\,-\,1)$. From \cite{JK}, it follows that $H_{3,n}$ is irreducible. The complement $H_{3,n}\,\setminus\, H^{\circ}$ consists of subschemes with Hilbert function $(1,2)$. There is a bijective morphism
$$\textrm{Gr}(2,n)\,\longrightarrow\, H_{3,n} \,\setminus\, H^{\circ},$$
where $\textrm{Gr}(2,n)$ denotes the Grassmanian variety of $2$-dimensional subspaces of $\mathbb{C}^{n}$. Therefore, in the normalization $\widetilde{H_{3,n}}$, the complement of $H^{\circ}$ has codimension $\,\geq\, 2$. It follows that $\widetilde{H_{3,n}}$ has Picard rank $1$.  
\end{proof}
\begin{remark}\label{drum}
In fact, some more computation shows the following structure of $\widetilde{H_{3,n}}$: There is a birational map
\begin{equation*}
\mathbb{P}_{\mathbb{P}^{n-1}}(\mathcal{O}(1)\,\oplus\,\Omega)\,\longrightarrow\, \widetilde{H_{3,n}},   \end{equation*}
contracting the divisor $\mathbb{P}_{\mathbb{P}^{n-1}}(\Omega)$ to a copy of $\textrm{Gr}(2,n)$. In other words, $\widetilde{{H}_{3,n}}$ is the drum constructed upon the triple $(\mathbb{P}_{\mathbb{P}^{n-1}}(\Omega), \mathcal{O}_{\mathbb{P}^{n-1}}(3), \mathcal{O}_{{\textrm{Gr}}(2,n)}(1))$, in the language of \cite{ORSW} or \cite{bansal2024extremal}
\end{remark}
\begin{lemma}\label{rho}
    Under the assumptions of Theorem \ref{2}, we have $\rho(\textrm{Hilb}^{m}X\,\setminus\, S^{m}(X))\,=\,1.$
\end{lemma}
\begin{proof}
Let $E$ be the exceptional divisor of  the Hilbert--Chow morphism 
$$\textnormal{Hilb}^{m}(X)\,\xrightarrow{\,\,\,p\,\,\,}\, S^{m} X.$$ 
To prove the claim, it suffices to show that $E$ is irreducible. For $n\,=\,2$, this is a classical result. For $m\,=\,2$, it follows by observing that $E\,=\,E^{\circ}$ is a $\mathbb{P}$--bundle over $W\,=\,W^{\circ}$, and hence irreducible. Therefore, we focus on the case $m\,=\,3$ and $n\,\geq\, 3$.

As mentioned above, $E^{\circ}$ is a smooth open subvariety of $E$, so its closure $\overline{E^{\circ}}$ is the unique irreducible component of $E$ intersecting $E^{\circ}$. To conclude that $\overline{E^{\circ}}\,=\, E$, it suffices to show that for any point $[Z]\,\in\,E\,\setminus\, \overline{E^{\circ}}$ lies in an irreducible subset of $E$ that contains $[Z]$ and intersects $E^{\circ}$. 

Since $[Z]\,\in\, E\,\setminus\,\overline{E^{\circ}}$, $Z$ is supported on a single point $x\,\in\, X$. As $Z$ has length $3$, there exists a smooth, locally closed subvariety $S\,\subset\,X$ of dimension $2$ such that $Z$ is a subscheme of $S$. This gives a locally closed embedding 
\begin{equation*}
\textrm{Hilb}^{3}S\,\xlongrightarrow{j}\,\textrm{Hilb}^{3}(X).
\end{equation*}
Let $F\,\hookrightarrow\, \textrm{Hilb}^{3}(S)$ be the exceptional divisor for the Hilbert--Chow morphism
\begin{equation*}
\textrm{Hilb}^{3}(S)\,\longrightarrow\, S^{3}(S).
\end{equation*}

Then $F$ is irreducible, and $j(F)\,\subset\, E$. Moreover, $[Z]\,\in\,j(F)$ and $j(F)\,\cap\, E^{\circ}\,\neq\,\emptyset$. This completes the proof.
\end{proof}

The following three lemmas are stated under the assumption that $X$ is a smooth surface and $m\,\geq\, 4$, using the notations $E_{\pi}$ and $E_{\pi}^{\circ}$ from \S2.
\begin{lemma}\label{local}
Let $\pi$ be any partition of $m$. Then $E_{\pi}$ is  irreducible of dimension $m\,+\,|\pi|$.    
\end{lemma}
\begin{proof}
This follows from \cite[Section 3.6]{DMM}.    
\end{proof}
\begin{lemma}\label{Sing}
We have
$$\emph{Sing}\, E\,=\,E_{(2,2,\underline{1})}\,\cup\,E_{(3,\underline{1})}.$$
\end{lemma}
\begin{proof}
As $E^{\circ}$ is a $\mathbb{P}^{1}$--bundle over $W^{\circ}$, we have 
$$\textrm{Sing}\,E\,\subset\, E\,\setminus\,E^{\circ}\,=\, E_{(2, 2, \underline{1})}\,\cup\, E_{(3,\underline{1})}.$$

To prove the reverse inclusion, it suffices to show that 
$$E^{\circ}_{(2,2,\underline{1})}\,\subset\, \textrm{Sing}\,E \quad and \quad E^{\circ}_{(3,\underline{1})}\,\subset\, \textrm{Sing}\,E,$$
since $E_{\pi}\,=\, \overline{E_{\pi}^{\circ}}$ for all $\pi$.

This is a local analytic question, so we may assume $X\,=\, \mathbb{A}^{2}$. For any positive integer $a$, and a length $a$ subscheme $q$ of $\mathbb{A}^{2}$ supported at the origin, let 
$$V_{a,q}\,=\,\langle \textrm{Hilb}^{a}(\mathbb{A}^{2}), q \rangle\,\cong\, \langle\mathbb{A}^{2a},0 \rangle.$$

For $a\,\geq\, 2$, let  $E^{g}_{a, q}\,\subseteq\, V_{a,q}$ denote the germ of the exceptional divisor of the Hilbert-Chow morphism
$$\textrm{Hilb}^{m}(\mathbb{A}^{2})\,\to\, S^m (\mathbb{A}^{2})$$
at the point $q$. 

Now, for a point $q^{\prime}\,\in\, \textrm{Hilb}^{m}(\mathbb{A}^{2})$, we may write $q^{\prime}$ as
$$q^{\prime}\,=\, \bigsqcup_{i}(q_{i}\,+\, x_{i}),$$
where each $q_{i}$ is a length $a_{i}$ subscheme supported  at the origin and $x_{1},\ldots, x_{r}$ are distict points in $\mathbb{A}^{2}$. Here $q_{i}\,
+\,x_{i}$ denotes the translation of $q_{i}$ by $x_{i}$. By a similar proof to {\cite[Lemma 2.2]{BSV3}}, there is an isomorphism of germs:
\begin{equation}\label{germ}
\phi\,:\,\prod_{i=1}^{r}V_{a_{i}, q_{i}}\,\longrightarrow\, \langle\textrm{Hilb}^{m}X, q^{\prime} \rangle   
\end{equation}
such that:
\begin{equation*}
\phi^{-1}\langle E,q^{\prime} \rangle\,=\,\bigcup_{\substack{1 \leq i \leq r \\ a_i \geq 2}}\,(V_{a_{1}, q_{1}}\,\times\, \ldots\, \times\, V_{a_{i-1}, q_{i-1}}\,\times\, E^{g}_{a_{i},q_{i}} \,\times\, V_{a_{i+1}, q_{i+1}} \,\times\,\ldots\,\times V_{a_{r},q_{r}}).    
\end{equation*}
Let $q^{\prime}\,\in\, E^{\circ}_{(2,2,\underline{1})}$. By \eqref{germ}, $\phi^{-1}\langle E, q^{\prime} \rangle$ has two distinct components, so $\langle E, q^{\prime} \rangle$ is not analytically irreducible. Thus, 
$$E^{\circ}(2,2, \underline{1})\,\subseteq\, \textrm{Sing}\,E.$$ 

To prove that $E^{\circ}_{(3,\underline{1})}\,\subset\, \textrm{Sing}\,E$, by \eqref{germ} it suffices to show that $E^{g}_{3,q}$ is singular for all subscheme $q$ of $\mathbb{A}^{2}$ supported at the origin. That is, we need to show 
$$E_{(3)}\,\subset\, \textrm{Sing}\,E \quad \textrm{for}\,\, m\,=\,3$$.

Assume $m\,=\,3$. Suppose, for contradiction, that $E_{(3)}\,\nsubseteq\,\textrm{Sing}\,E$. Since $E_{(3)}$ is an irreducible codimension $1$ subvariety of $E$ and $E\,\setminus\, E_{(3)}\,=\, E^{\circ}$ is smooth, it would follow that 
$$\textrm{codim}_{E}(\textrm{Sing}\,E)\,\geq\,2.$$
However, from \cite{BEG} or \cite[Theorem B]{Sta}, we know that $E$ is a free divisor in $\textrm{Hilb}^{3}\,X$, i.e., the sheaf of logarithmic differentials $\Omega(log E)$ is locally free. By {\cite[Section 2]{aleksandrov1990nonisolated}} or \cite[Section 1.4]{BEG}, $E$ is either smooth or codim${_{E}}(\textrm{Sing}\,E)\,=\,1$. As we have shown codim${_{E}}(\textrm{Sing}\,E)\,\geq\,2$, $E$ must be smooth. We will now derive a contradiction by showing that  that $E$ is singular at the point $q$ corresponding to the subscheme $\textrm{Spec}\,\frac{\mathbb{C}[X,Y]}{(X,Y)^{2}}$. We show in fact $\textrm{dim}\,T_{q}E_{(3)}\,=\, 6$, which suffices as dim$\,E\,=\, 5$.

Let 
$$g\,:\,E_{(3)}\,\longrightarrow\,W_{(3)}\,\cong\,X$$

be the restriction of $p$. Let $H$ denote the Hilbert scheme of length $3$ subschemes of $\mathbb{A}^{2}$ supported at the origin. By \cite[Section 4]{Br}, $H$ is isomorphic to the projective cone over the twisted cubic curve and $q\, \in\, H $ is the singular point. So, dim $ T_{q}H\,=\,4$. 
Finally, note that $g$ is a locally trivial fibre bundle with each fibre isomorphic to $H$. It follows that $\textrm{dim } T_{q}E_{(3)}\,=\,\textrm{dim } T_{q}H \,+\, \textrm{dim} \hspace{1pt} X\,=\, 4\,+\,2\, =\,6.$    
\end{proof}

\begin{lemma}\label{capcup}
For integers $t$ such that $3\,\leq\,t\,\leq\, m\,-\,2$, we have:
$$E_{(t, \underline{1})}\,\cap\, E_{(t-1, 2, \underline{1})}\,=\, E_{(t+1,\underline{1})}\,\cup\, E_{(t,2,\underline{1})}.$$
Also,
$$E_{(m-1,1)}\,\cap\, E_{(m-2, 2)}\,=\, E_{(m)}.$$   
\end{lemma}
\begin{proof}
Since each $E_{\pi}$ is the preimage $p^{-1}(W_{\pi})$ under the Hilbert-Chow morphism $p$, it is enough to prove:
\begin{itemize}
\item[$(a)$] $W_{(t,\underline{1})}\,\cap\, W_{(t-1,2,\underline{1})}\,=\, W_{(t+1,\underline{1})}\,\cup\, W_{(t,2,\underline{1})},$\quad\textrm{for all}\,\, $3\,\leq\,t\,\leq m-2$, 
\item[$(b)$] $W_{(m-1,1)}\,\cap\,W_{(m-2,2)}\,=\, W_{(m)}$.
\end{itemize}
Part (b) follows directly from {\cite[Lemma 2.1(v)]{BSV3}}.

To prove part $(a)$, it suffices to show: 

If a partition $\pi$ satisfies $\pi\,>\,(t,\underline{1})$ and $\pi\,>\,(t-1, 2, \underline{1})$, then either $\pi\,\geq\,(t+1,\underline{1})$ or $\pi\,\geq\, (t, 2, \underline{1})$.

Let $\pi\,=\, (a_{1},\ldots,a_{r})$. Since $\pi\,>\, (t, \underline{1})$, we have $a_{1}\,\geq\, t$. If $a_{1}\,\geq\, t+1$, then $\pi\,\geq\, (t+1,\underline{1})$. If $a_{1}\,=\, t$, then $a_{2}\,\geq\, 2$ as $\pi\,\neq\, (t,\underline{1})$. So, $\pi\,\geq\,(t, 2, \underline{1})$. This completes the proof of the Lemma.
\end{proof}

\textit{Proof of Theorem \ref{descent}:}

\begin{proof}
    Let $\textnormal{Hilb}^{m}(X)\,\xrightarrow{\,\,\,p\,\,\,}\, S^{m} X$ be the Hilbert-Chow morphism. We consider the three cases separately. 

\underline{\textrm{Case}\,1}: $m\,=\,2.$

If $X\, =\, \mathbb{P}^{2}$, the result follows from \cite{belmans2020automorphisms}.
Assume now that $X\,\not\cong\, \mathbb{P}^{2}$. We identify $X$ as a subvariety of $S^{2} X$ under the diagonal embedding. Then $p(E)\, = \,X$, and the restriction $p_{{|}_{E}}\, :\, E\, \longrightarrow\, X$ identifies $E$ with $\mathbb{P}_{X}(\Omega_{X})$, the projectivized cotangent bundle over $X$. Since $\phi_{{|}E}\, :\, E\,\longrightarrow\, E$ is an automorphism, by Corollary \ref{Pomega}, it must be over an automorphism of $X$. Therefore, any curve contracted by $p$ must also be contracted by $p\,\circ\, \phi$, which means $\phi$ descends to an automorphism of $S^2X.$

\underline{\textrm{Case}\,2}: $m\,=\,3.$

If $X\,=\, \mathbb{P}^{2}$, the result again follows from \cite{belmans2020automorphisms}. Now assume $X\,\not\cong \,\mathbb{P}^{2}$. Fix a point $w\,\in\,W_{(3)}$. By Lemma \ref{Hilbert drum}, the fiber $p^{-1}(w)$ is irreducible of dimension $2(n-1)$, and its  normalization $D$ has Picard rank $1$. Let $p_{1}\,:=\, (p\,\circ\,\phi)|_{D} \,:\,D\,\xlongrightarrow\, W$ be the restriction to $D$. Since $D$ is normal and the normalization of $W$ is isomorphic to $X\,\times\,X$ by {\cite[Theorem 2]{BSV3}}, $p_{1}$ induces a morphism $p_{1}^{\prime}\,:\, D\,\longrightarrow\, X\,\times\,X$.

We now claim that $p_{1}^{\prime}$ is constant.

If $n\,\geq\,3$, then $2(n\,-\,1)\,>\,n$, so any morphism from $D$ to $X$ must be constant, since $\rho(D)\,=\,1$. If $n\,=\,2$, and there is a non-constant map $D\,\longrightarrow\,X$, it must be surjective as $\rho(D)\,=\,1$. But $D$ is rational, (in fact, $D$ is the weighted projective plane $\mathbb{P}(1,1,3)$ by Remark \ref{drum}), we must have $X$ is a rational surface of Picard rank $1$, hence $X\,\cong\,\mathbb{P}^{2}$; a contradiction. Thus, in all cases, any map $D\,\longrightarrow\, X$ must be constant, and so $p_{1}^{\prime}$ is constant. Hence $p_{1}$ is constant. That means $p^{-1}(w)$ is mapped to a point under $p\,\circ\,\phi$. By Lemma \ref{rho}, any curve contracted by $p$ must be contracted by $p\,\circ\,\phi$. So, $\phi$ descends to an automorphism of $S^2X.$ 

\underline{\textrm{Case} \,3}: $n\,=\, 2$.

Assume $m\,\geq\,4$, as otherwise we are done by cases 1 and 2. We now prove the following:

\underline{Claim}: $\phi(E_{(m)})\,=\, E_{(m)}$.

\underline{Proof of the claim}: Define a sequence of closed subsets $F_{t}$ of $E$ for $3\,\leq\, t\,\leq\, m$ as follows:
\begin{itemize}
\item Let $F_{3}\,=\,\textrm{Sing}\,E$, the singular locus of $E$
\item Define $F_{t+1}$ to be the intersection of the irreducible components of $F_{t}$ for each $3\,\leq t\,\leq\, m\,-\,1$.
\end{itemize}

Since $\phi(E)\,=\, E$, by induction $\phi(F_{t})\,=\, F_{t}$ for all $t$. In particular, $\phi(F_{m})\,=\, F_{m}$. On the other hand, from Lemma \ref{Sing},  we know that $F_{3}\,=\, E_{(2,2,\underline{1})}\,\cup\, E_{(3, \underline{1})}$. Suppose that for some $3\,\leq\, t\,\leq m\,-\,2$, we have $F_{t}\,=\, E_{(t,\underline{1})}\,\cup\, E_{(t-1, 2, \underline{1})}$.  Then both $E_{(t,\underline{1})}$ and $E_{(t-1,2, \underline{1})}$ are the irreducible components of $F_{t}$, as they are irreducible and have same dimension by Lemma \ref{local}. So, Lemma $\ref{capcup}$, tells us that:
$$F_{t+1}\,=\, E_{(t+1,\underline{1})}\,\cup\, E_{(t, 2, \underline{1})}.$$ 
Continuing this inductively, we find $F_{t}\,=\, E_{(t,\underline{1})}$ for all $3\,\leq\,t\,\leq m\,-\,1$. So, $F_{m}\,=\, E_{(m)}$ by Lemma \ref{capcup}. Therefore, $\phi(E_{(m)})\,=\, E_{(m)}$. This completes the proof of the claim.

Now we complete the proof. By Lemma \ref{rho}, it suffices to show that some curve contracted by $p$ is also contracted by $p\,\circ\,\phi$. Suppose not. Let $H$ be a fiber of  the map $E_{(m)}\,\xlongrightarrow{p}\,W_{(m)}$. Then the composition $H\,\xlongrightarrow{p\circ\phi}\,S^{m}X$ must be finite. However, by the claim, $p\,\circ\,\phi(H)\,\subset\, W_{(m)}$. Also, $\text{dim}(H)\,=\, m\,-\,1$, by Lemma \ref{local} and $\textrm{dim }(W_{(m)})\,=\, 2$, by {\cite[Theorem 2]{BSV3}}. Since $m-1\,>\,2$, the map $p\circ \phi$ cannot be finite, a contradiction. 
\end{proof}
\begin{remark}
Note that we have also shown in the proof that for any automorphism $\phi$ of $\textrm{Hilb}^{m}X$, if $\phi$ preserves the small diagonal $E_{(m)}$, then $\phi$ is natural.
\end{remark}
\begin{remark}\label{determine}
    The same proof also shows that for $m,n$ as in Theorem \ref{2}, and for smooth projective varieties $X$ and $Y$ of dimension $n$, any isomorphism $\textrm{Hilb}^{m}X\,\to\, \textrm{Hilb}^{m}Y$ that carries the big diagonal to the big diagonal descends to an isomorphism $S^mX\,\to\, S^mY$, and hence $X\cong Y$ by {\cite[Corollary 4.2]{BSV4}}. Thus, the pair of $\textrm{Hilb}^{m}X$ and the big diagonal determines $X.$
\end{remark}

\textit{Proof of Theorem \ref{lift}:}

\begin{proof}
    
This follows from \cite[Proposition 2]{Gir} and the following Lemma.
\end{proof}
\begin{lemma}\label{equivariant}
Let $X$ be a smooth variety of dimension $\geq 2$. Let $m\geq 2$ be an integer and $\psi\in\textrm{Aut}(S^mX)$. Then $\psi$ lifts to an $S_m$--equivariant automorphism of $X^m.$ 
\end{lemma}
\begin{proof}
The proof of {\cite[Proposition 9 and 12]{belmans2020automorphisms}} is valid for every smooth variety $X$ of dimension $\geq 2$, not only surfaces. This shows $\psi$ lifts to an automorphism $\phi$ of $X^m.$ Hence there is a group automorphism $\tau$ of $S_m$ such that $\phi(\underline{x}_{\pi})\,=\,\phi(\underline{x})_{\tau(\pi)}$ for all $\underline{x}\,\in\, X^m$ and $\pi\,\in\, S_m.$ So, if $Fix(\pi)\,=\,\{\underline{x}\,\in\, X^m | \underline{x}_{\pi}\,=\,\underline{x}\}$, then $\phi(\textrm{Fix}(\pi))\,=\,\textrm{Fix}(\tau(\pi))$ for each $\tau\,\in\, S_m.$ In particular, $\textrm{Fix}(\pi)$ and $\textrm{Fix}(\tau(\pi))$ has same dimension. This shows that $\tau$ preserves cycle structures, hence $\tau$ is an inner automorphism. So, composing $\phi$ by a permutation if necessary, we get $\phi$ is $S_m$--equivariant.
\end{proof}

\section{Automorphisms of symmetric power of Higher dimensional varieties}

In this section, we prove Theorem \ref{3}$(i)$. The idea of the proof is similar to that in \cite{Ha}: first, by Lemma \ref{equivariant}, any automorphism of $S^mX$ lifts to an $S_m$--equivariant automorphism of $X^m$. So, we need to determine the $S_m$--equivariant automorphisms of $X^m$. Unlike in \cite{Ha}, we must analyze this last point more carefully, as we do not assume $X$ to be indecomposable. 

We will need the following Lemma to describe the group $\textrm{Aut}(X^m)$.
\begin{lemma}\label{autoproduct}
Let $r, a_{1},\ldots, a_{r}$ be positive integers, and let $Y_{1},\ldots, Y_{r}$ be pairwise non-isomorphic indecomposable normal projective varieties. Suppose that either $H^{1}(Y_{i}, \mathcal{O}_{Y_{i}}) = 0$ for all $i$, or that each $Y_i$ is smooth with $K_{Y_i}$ ample. Let $Y\, =\, \prod _{i} Y_{i}^{a_{i}}$. Then, 
$$\textnormal{Aut}(Y)\, \cong \,\prod_{i}(\textnormal{Aut}(Y_{i})^{a_{i}}\,\rtimes\, S_{a_{i}}).$$
\end{lemma}
\begin{proof}
If all $Y_i$'s are smooth with ample canonical class, then the result follows from the proof of {\cite[Theorem 4.2]{BHPS}}. If $H^{1}(Y_{i}, \mathcal{O}_{Y_{i}}) \,=\, 0$ for all $i$, then the same proof as {\cite[Theorem 2.7, Corollary 2.8]{Ha}} proves the lemma (see also {\cite[Theorem 4.1]{Og}}). 
\end{proof}

\textit{Proof of Theorem \ref{3} $(i)$ :}
\begin{proof}
 Write $X\,=\,\prod_{i=1}^k Z_i$, where $Z_i$'s are indecomposable. Define a group homomorphism $S_2^k\,\xrightarrow{h}\, \textnormal{Aut}(X^{2})$ by 
 $$ h(\pi_{1},\pi_{2},\ldots,\pi_{k})((z_{ij}))\,=\, ((z_{\pi_j(i),j} )),$$ 
 where elements of $X^{2}$ are written as matrices $ ((z_{ij}))_{\underset{1 \leq j \leq k}{1 \leq i \leq 2}}$, with $ z_{ij} \in Z_{j}.$ This induces a group homomorphism $S_2^k\,\xrightarrow{\Bar{h}}\, \textnormal{Aut}(S^2X)$ with kernel generated by $(\pi,\ldots,\pi)$, where $\pi$ is the transposition in $S_2$. So, $\textnormal{im}(\Bar{h})\,\cong\, (\mathbb{Z}/2\mathbb{Z})^{k-1}$. For example, when $k\,=\,2$, the non-trivial element of $\textnormal{im}(\Bar{h})$ is the following automorphism of $S^2X$:
\begin{equation}\label{swap}
    (z_{1,1}, z_{1,2})\,+\,(z_{2,1}, z_{2,2})\,\mapsto\, (z_{1,1}, z_{2,2}) \,+\, (z_{2,1}, z_{1,2}).
\end{equation} For $m\,=\,2$, we shall prove the more precise statement that $\textnormal{Aut}(S^{2}(X))\, = \,\textnormal{im}(\Bar{h}).$
 
Let $\psi\,\in\, \textnormal{Aut}(S^{m}X)$. By Lemma \ref{equivariant}, $\psi$ lifts to $S_m$--equivariant $ \phi \,\in\, \textnormal{Aut}(X^{m})$. By composing $\psi$ and $\phi$ with natural automorphism assume $\phi|_{\Delta_X}\,=\,id.$ Our goal is to show:
\begin{enumerate}
    \item If $m\,\geq\, 3$, then $\phi\,=\,id$,
    \item If $m\,=\,2$, then $\phi\,\in\, \textnormal{im}({h}).$
\end{enumerate}

Write $X\,=\,\prod_{i=1}^r X_i$, $X_i\,=\,Y_i^{a_i}$, where $Y_i$'s are indecomposable, pairwise non-isomorphic and $a_i\,>\,0$. Clearly, the hypothesis of Lemma \ref{autoproduct} is satisfied. By Lemma \ref{autoproduct}, there are $\phi_i\,\in\, \textnormal{Aut}(X_i)$ such that $\phi\,=\,\prod_i \phi_i.$ Note that $\phi_i$ is also $S_m$-equivariant and $\phi_i|_{\Delta_{X_i}}\,=\,id.$ It suffices to show $(1)$ and $(2)$ for each $\phi_i$, so without loss of generality assume $k\,=\,1.$ So, $X\,=\,Y^a$ for indecomposable $Y.$

We write elements of $X^m=Y^{ma}$ as a matrix $((y_{i,j}))_{\underset{1\leq j\leq a}{1\leq i \leq m} }$, where $y_{i,j}\,\in\, Y.$ Regard $S_{ma}$ as the symmetric group on the set $ [m]\, \times\, [a] .$ By Lemma \ref{autoproduct}, there are $\phi_{i,j}\in \textnormal{Aut}(Y) $ and $\sigma\,\in\, S_{ma}$ such that $\phi((y_{i,j}))\,=\,((\phi_{i,j}(y_{\sigma(i,j)})))$.

Write $\sigma(i,j)\,=\,(\sigma_1(i,j), \sigma_2(i,j))\,\in \,[m]\times [a].$ Fix any $y_1, y_2,\ldots,y_a\in Y.$ Set $y_{i,j}\,:=\,y_j$ for all $(i,j)\,\in\, [m]\,\times\, [a]$. Evaluating $\phi$ at the element of $\Delta_X$ corresponding to $x\,=\,(y_1,\ldots,y_a)\,\in\, X$ and using $\phi|_{\Delta_X}\,=\,id,$ we get $\phi_{i,j}(y_{\sigma_2(i,j)})\,=\,y_j$ for all $i,j$ and $y_1,\ldots,y_a\,\in\, Y$. This forces $\sigma_2(i,j)\,=\,j$ and $\phi_{i,j}\,=\,id$ for all $i,j.$

So, if we write $\pi_j(i)\,=\,\sigma_1(i,j)$, then $\pi_j\,\in\, S_m$ for $1\,\leq\, j\,\leq\, a$, and $\phi((y_{i,j}))\,=\,((y_{\pi_j(i),j}))$ for all $y_{i,j}\,\in\, Y.$ Now the $S_m$--equivariance of $\phi$ implies that $\pi_j\,\in\, \mathcal{Z}(S_m)$ for all $j$. This proves both $(1)$ and $(2).$
\end{proof}

 \section{Automorphisms of symmetric power of surfaces}

We will prove Theorems \ref{4} and \ref{3}$(ii)$ and Corollary \ref{c1} in this section. We need the following lemmas.

\begin{lemma}\label{infinite}
Let $X$ be an abelian surface with $\emph{End}(X)^{\times}$ infinite. Then for each $m\,\geq\,2$, the symmetric power $S^{m}X$ admits a non-natural automorphism.
\end{lemma}
\begin{proof}
Since $\textrm{End}(X)$ is finitely generated as an abelian group and $\textrm{End}(X)^{\times}$ is an infinite, there exist $u_1, u_2\,\in\, \textrm{End}(X)^{\times}$, with $u_1\,\neq\, \pm u_2$, such that $m| (u_1-u_2)$ in $\textrm{End}(X)$. Letting $\alpha\,=\,u_1^{-1}u_2\,\neq\, \pm 1$, we get $m|(1\,-\,\alpha)$ in $\textrm{End}(X)$. 

The rest of the proof follows exactly as in \cite{Sas}. More explicitly, define
$$f\,=\, \frac{1\,+\,(m\,-\,1)\alpha}{m},\quad g\,=\, \frac{1\,-\,\alpha}{m}\,\in\, \textrm{End}(X).$$ 

Then, $f, g\,\neq\, 0.$ Consider the automorphism  $\Phi$ of $X^{m}$ given by the matrix 
\[
\begin{bmatrix}
f & g & \cdots & g \\
g & f & \cdots & \vdots \\
\vdots & \vdots & \ddots & \vdots \\
g & \cdots & g & f
\end{bmatrix}.
\]
This descends to a non-natural automorphism of $S^{m}X$. 
  
\end{proof}

\begin{lemma}\label{unit}
Let $X$ is an abelian surface such that $\emph{End}_{\mathbb{Q}} (X)$ commutative, and let $m\,\geq\,2$ be an integer. Then $S^{m}X$ admits a non-natural automorphism if and only if there exists  $\alpha$ with $\alpha\,\neq\, \pm 1$ in $\textrm{End}(X)^{\times}$, such that $m|(1\,-\,\alpha)$ in $\textrm{End}(X)$.  
\end{lemma}
\begin{proof}
$(\Leftarrow)$ The same arguement as in Lemma \ref{infinite} applies.
  
$(\Rightarrow)$ Suppose $\phi\,\in\, \textrm{Aut}(S^{m}X)$ is non-natural. Then by Lemma \ref{equivariant}, $\phi$ lifts to $\Phi\,\in\, \textrm{Aut}(X^{m})$. Composing $\phi$ with a natural automorphism if necessary, we may assume that $\Phi|_{\Delta_X}\,=\,id.$ In particular, $\Phi(\underline{0})\,=\,\underline{0}$, so $\Phi\,\in\, \textrm{End}(X^m)^{\times}$.

As in \cite{Sas}, by precomposing $\phi$ and $\Phi$ with an element of $S_m$, we can write $\Phi(\underline{x})\,=\, P\underline{x}$ for some $m\,\times\,m$ matrix 
\[ P=
\begin{bmatrix}
f & g & \cdots & g \\
g & f & \cdots & \vdots \\
\vdots & \vdots & \ddots & \vdots \\
g & \cdots & g & f
\end{bmatrix},
\]
with $f, g\,\in\, \textrm{End}(X)$, $f, g\,\neq\,0$. Taking the determinant, we see that
$$(f\,-\,g)^{m-1}(f\,+\,(m\,-\,1)g)\,\in\, \textrm{End}(X)^\times.$$
Also, since $\Phi|_{\Delta_X}\,=\,id$, we have $f\,+\,(m\,-\,1)g\,=\,1.$ Let $\alpha\,=\,f\,-\,g\,\in\, \textrm{End}(X)^\times$. Now $f\,+\,(m\,-\,1)g\,=\,1$ shows   $mg\,=\, 1-\alpha.$ 
Since $f,g\,\neq\,0,$ we get $\alpha\,\neq\, \pm 1$.

\end{proof}
\begin{corollary}\label{nounit}
Let $E$ be an elliptic curve and suppose $\phi\,\in\, \emph{Aut}(E^m)$ descends to $S^m E$ and satisfies $\phi|_{\Delta_E}\,=\,$ id. Then $\phi$ is a permutation of the factors.
\end{corollary}
\begin{proof}
Same proof as in the $(\Rightarrow)$ direction of Lemma \ref{unit}, noting that 
$\textrm{End}(E)^\times=\{\pm 1\}, \\
\{\pm 1,\, \pm i\}$ or $\{\pm 1,\, \pm \omega,\, \pm \omega^{2} \}$, where $\omega\,=\, e^{2\pi i/3}$ and so in any case, $\alpha$ as in Lemma \ref{unit} does not exist.
\end{proof}
\begin{lemma}\label{albanese}
Let $Y$ be a minimal smooth projective surface of Kodaira dimension $1$. If the Albanese dimension $a(Y)\,=\,2$, then $Y$ contains no rational curve.    
\end{lemma}

\begin{proof}
Let $g\,:\,Y\,\xlongrightarrow\,A$ be the Albanese map, and let $p\,:\,Y\,\longrightarrow\,B$ be the Iitaka fibration. Suppose $C\,\subset\, Y$ is a rational curve. By \cite[Theorem 1.1]{Lu}, we have $K_{Y}\cdot C\,=\, 0$. Hence, $p(C)\,=\, b_{0}$, a point in $B$. 

By Kodaira's classification of singular fibres of the minimal elliptic fibration $p$, the fibre $p^{-1}(b_{0})$ is a chain of rational curves. Thus $g(p^{-1}(b_{0}))$ is a point. For each $b\,\in\,B$, the fiber $p^{-1}(b)$ is numerically equivalent to $p^{-1}(b_{0})$, so $g(p^{-1}(b))$ is also a point. This contradicts $a(Y)\,=\,2$.
\end{proof}

\begin{lemma}\label{equivalence}
Let $X$ be a smooth projective surface, and let  $C$ and $F$ be smooth projective curves. Suppose $X\,\xlongrightarrow{p}\,C$ is an isotrivial contraction with general fibre $F$, and  $H^1(X, \mathcal{O}_X)\,\neq \,0$. Assume further that $X$ is not an abelian surface. Then the following are equivalent:
\begin{enumerate}
\item[(a)] There is a non-constant map $C\,\longrightarrow\, \textrm{Aut}_{\,C}X$. 
\item[(b)] The surface $X$ lies in class $\mathcal{C}$, the map $p$ is the Iitaka fibration of $X$, and there is a non-constant map $C\,\longrightarrow\, F$.
\end{enumerate}
\end{lemma}

\begin{proof}
It is clear that
$(a)$ is equivalent to the following

$(a)^{\prime}:$ There is a non-constant map $C\,\longrightarrow\,\textrm{Aut}_{C}^0X$.

$\underline{(a)^{\prime}\,\Rightarrow\,(b):}$ If $F\,=\, \mathbb{P}^{1}$, then $\textrm{Aut}^{0}_{C}X$ has no positive dimensional complete subvariety, see \cite{Mar}. So, $F\,\not \cong \, \mathbb{P}^{1}$. Since $\textrm{Aut}^{0}X$ is non-trivial, a result in \cite{Fon} implies that $X$ is either a hyperelliptic surface or $X$ lies in class $\mathcal{C}$. In both cases, $\textrm{Aut}^{0}X$ is an elliptic curve, so we must have $\textrm{Aut}^{0}_{C}X \,=\,\textrm{Aut}^{0}X$, as $\textrm{Aut}^{0}_{C}X$ is non-trivial. If $X$ is hyperelliptic surface, as $\textrm{Aut}^{0}X$ preserves the fibres of $p$, $p$ must be a contraction to $\mathbb{P}^1$. So, $C\,=\,\mathbb{P}^1$, but then $(a)$ is impossible. So, $X$ is in class $\mathcal{C}$. As $\textrm{Aut}^{0}X$ preserves the fibres of $p$, we see that $p$ must be the Iitaka fibration of $X$, $\textrm{Aut}^{0}_{C}(X)$ is an elliptic curve isogeneous to $F$. So, $(b)$ holds.

$\underline{(b)\,\Rightarrow\,(a)^{\prime}:}$ By \cite{Fon}, $\textrm{Aut}^{0}_{C}(X)$ is an elliptic curve isogenous to $F$. So $(a)^{\prime}$ holds.
\end{proof}

Now we prove Theorem \ref{4}. Let us first show that in cases $(1)-(4)$ there is indeed a non-natural automorphism.

 \begin{enumerate}
     \item Equation \ref{swap} gives a non-natural automorphism. 
     \item Note that $\textrm{GL}(2, \mathbb{Z})\,\subset\, \textrm{End}(E^{2})^{\times},$ so $\textrm{End}(E^{2})^{\times}$ is infinite. So, $\textrm{End}(X)^{\times}$ is infinite, and we are done by Lemma \ref{infinite}.
     \item By Lemma \ref{infinite}, it suffices to show that $\textrm{End}(X)^\times$ is infinite. Let $K=\textrm{End}_{\mathbb{Q}}(X).$ If $K$ is not a field, then $K$ is an indefinite quaternion divison algebra over $\mathbb{Q}$, so $\textrm{End}(X)^\times$  is infinite by \cite[Appendix]{GV}. If $K$ is a totally real field or $[K:\mathbb{Q}]\,>\, 2$, then $\mathcal{O}^{\times}_{K}$, hence $\textrm{End}(X)^\times$ is infinite by Dirichlet's unit theorem.

     \item Let $p:X\to C$ be the Iitaka fibration of $X$. By Lemma \ref{equivalence}, we have a non-constant map $C\,\longrightarrow\, \textrm{Aut}^{0}_{C}X$. As $\textrm{Aut}^{0}_{C}X$ is an elliptic curve by \cite{Fon}, we get a surjective map $\textrm{Jac}(C)\,\longrightarrow\, \textrm{Aut}^{0}_{C}X$ of abelian varieties. We have the Abel-Jacobi map $S^{m-1}C\,\longrightarrow\, \textrm{Jac}(C)$ whose image generates $\textrm{Jac}(C)$ as an abelian variety. So, the image of the composition $S^{m-1}C\,\longrightarrow\, \textrm{Aut}^{0}_{C}X$ generates $\textrm{Aut}^{0}_{C}X$ as an abelian variety; in particular, it is non-constant. Denote this composition by $\alpha\,\longmapsto\, \psi_{\alpha}^{\prime}$, where $\alpha\,\in\, S^{m-1}C$, $\psi_{\alpha}^{\prime}\,\in\,\textrm{Aut}^{0}_{C}X$. So, $p\,\psi^{\prime}_{\alpha}(x)\,=\, p(x)$ for all $x\,\in\,X$.
     
     Define $T\,:\,X\,\longrightarrow\, X$ by $T(x)\,=\, \psi^{\prime}_{(m-1)p(x)}(x)$. We have $p\,T(x)\,=\, p(x)$ for all $x\,\in\,X$, and for each $c\,\in C$, $T_{|_{X_{c}}}\,:\, X_{c}\,\longrightarrow\, X_{c}$ is same as $\psi^{\prime}_{(m-1)c}$, hence an isomorphism. So, $T\,\in\, \textrm{Aut}_{C}X$. Now define a map 
\begin{equation*}
S^{m-1}C\,\longrightarrow\, \textrm{Aut}_{C}X    
\end{equation*}
given by 
\begin{equation*}
\alpha\,\longmapsto\,  \psi_{\alpha}\,:=\, (\psi^{\prime}_{\alpha})^{-1}\,\circ\, T.   
\end{equation*}
Clearly, this map is non-constant, and $\psi_{(m-1)p(x)}(x)\,=\,x$. Now define $\phi\,:\, X^{m}\,\longrightarrow\, X^{m}$ by 

\begin{equation*}
\phi((x_{i})_{i})\,=\, \left(\psi_{{\sum_{j\neq i}p(x_{j})}}(x_{i})\right)_{i}. 
\end{equation*}
So, $\phi_{|_{\Delta_{X}}}\,=\, {id}$. Note that $\phi$ is a morphism over $C^{m}$. Over $\underline{c}\,\in\, C^{m}$, $\phi\,:\, \Pi_{i} X_{c_{i}}\,\longrightarrow\, \Pi_{i}X_{c_{i}}$ is just the product map $
\prod_i \left( \left. \psi_{{\sum}_{j \ne i} c_j} \right|_{X_{c_i}} \right)$, which is an isomorphism. So, $\phi$ is an automorphism of $X^{m}$. As $\phi$ is $S_{m}$-equivariant, it induces an automorphism $\overline{\phi}$ of $S^{m}X$. As $\left. \phi \right|_{\Delta_{X}}\,=\, {id}$, if $\overline{\phi}$ were natural then $\overline{\phi}$ must be identity. So, $\phi$ has to be a permutation of coordinates. It forces $\psi_{\alpha}\,=\, {id}$ for all $\alpha$. But this is a contradiction as $\alpha\,\longmapsto\, \psi_{\alpha}$ is non-constant. So, $\overline{\phi}$ is non-natural.

 \end{enumerate}

 Now we proceed to show the converse. Assume $(1)-(4)$ do not hold. We need to show that every automorphism of $S^m X$ is natural. This will follow from the following theorems  \ref{isotrivial}, \ref{non-isotrivial} and \ref{nonminimal}, as follows.  If the Albanese dimension $a(X)\,=\,0$, then $H^1(X, \mathcal{O}_X)=0$, so we are done by Theorem \ref{3} $(i)$. So suppose $a(X)\neq 0.$ If $a(X)\,=\,1$, then letting $p$ be the contraction part of the Stein factorization of the Albanese map of $X$, we are done by Theorems \ref{isotrivial} and \ref{non-isotrivial}. So assume $a(X)\,=\,2$. 
 
 By \cite{belmans2020automorphisms}, we may assume $X$ is not of general type. If $X$ is a simple abelian surface then $\textrm{End}_{\mathbb{Q}}(X)$ is  $\mathbb{Q}$ or an imaginary quadratic extension of $\mathbb{Q}$, since $(3)$ does not hold. By Lemma \ref{unit} every automorphism of $S^m X$ is natural, noting that $\textrm{End}(X)^\times\,=\,\{\pm 1\}, \{\pm 1,\, \pm i\}$ or $\{\pm 1,\, \pm \omega,\, \pm \omega^{2} \}$, where $\omega\,=\, e^{2\pi i/3}$. So assume $X$ is not a simple abelian surface. If $X$ is an abelian surface isogenous to $E_1\,\times\, E_2$, where $E_1$ and $ E_2$ are non-isogenous elliptic curves, then $X$ has a smooth isotrivial elliptic fibration $p$ over an elliptic curve such that the base and the fibre are not isogenous, so hypothesis $(i)$ of Theorem \ref{isotrivial} is satisfied. So we are done by Theorem \ref{isotrivial}. As $(2)$ does not hold, we can assume that $X$ is not an abelian surface. If $X$ is non-minimal and birational to an abelian surface, we are done by Theorem \ref{nonminimal}. So we can assume $X$ is not birational to an abelian surface.

  So, $X$ is not of general type, $a(X)\,=\,2$, and $X$ is not birational to an abelian surface. By Enriques-Kodaira classification of surfaces, we see that the Kodaira dimension $\kappa(X)\,=\,1.$ If $X$ is minimal, letting $p$ to be the Iitaka fibration of $X$ we are done by Theorems \ref{isotrivial} and \ref{non-isotrivial}. If $X$ is not minimal, let $Y$ be the minimal model of $X$. So, $\kappa(Y)\,=\,1$ and $a(Y)\,=\,2$. By Lemma \ref{albanese}, $Y$ has no rational curve. So we are done by Theorem \ref{nonminimal}.

\begin{theorem}\label{isotrivial}
Let $X$ be a smooth projective surface, $C$, $F$ smooth projective curves. Let $X\,\xlongrightarrow{p}\,C$ be an isotrivial contraction with general fibre $F$. Assume that one of the following holds:
\begin{enumerate}
    \item[$(i)$] Every map $F\,\longrightarrow\, C$ is constant,
    \item [$(ii)$]$p$ is the contraction part of the Stein factorization of the Albanese map of $X$,
    \item[$(iii)$] $X$ is minimal and $p$ is the Iitaka fibration of $X$.
\end{enumerate}  Also assume that none of $(1)-(4)$ in Theorem \ref{4} holds.  Then every automorphism of $S^{m}X$ is natural for $m\,\geq\,2$. 
\end{theorem}

\begin{theorem}\label{non-isotrivial}
Let $X$ be a smooth projective surface, $C$ a smooth projective curve, $X\,\xlongrightarrow{p}\,C$ a non-isotrivial contraction. Assume either $(ii)$ or $(iii)$ of Theorem \ref{isotrivial} holds. Then for each $m\,\geq\,2$, every automorphism of $S^{m}X$ is natural.
\end{theorem}

\begin{theorem}\label{nonminimal}
Let $X$ be a non-minimal smooth projective surface with $k(X)\,\geq\,0$, and let $Y$ be its minimal model. Assume that $Y$ contains no rational curve. Then for each $m\,\geq\, 2$, any automorphism of $S^{m}X$ is natural.     
\end{theorem}

\textit{Proof of Theorem \ref{nonminimal}:}
\begin{proof}

Let $\phi\,\in\, \textrm{Aut}(S^{m}X)$.  By Lemma \ref{equivariant}, $\phi$ lifts to an automorphism $\Phi$ of $X^{m}$. By composing $\phi$ with a natural automorphism, we may assume that $\Phi|_{\Delta_X}= id$.

We have birational contraction $X\,\xlongrightarrow{f}\,Y$. Let $F\,=\,\{z_{1},\cdots,z_{r}\}\,\subset\,Y$ denote the fundamental locus of $f$, which is a finite nonempty set since $X$ is not minimal. We have induced birational contraction 
\begin{equation*}
X^{m}\,\xlongrightarrow{f_{(m)}}\, Y^{m},
\end{equation*}
whose fundamental locus $F_{m}$ has irreducible components 
\begin{equation*}
F_{i,j} \,=\, \pi_{i}^{-1}(z_{j}),\quad 1\,\leq\,i\,\,\leq m,  \quad 1\,\leq\,j\, \leq r, 
\end{equation*}
where $\pi_{i}\,:\, Y^{m}\,\longrightarrow\, Y$ denotes the projection onto the $i$-th factor. 

Since the fibres of $f_{(m)}$ are rationally chain-connected and $Y^m$ has no rational curve,  $\Phi$ descends to an automorphism $\Psi$ of $Y^{m}$. As $\Phi|_{\Delta_X}\,=\, id$, we have $\Psi|_{\Delta_Y}\,=\, id$. Clearly, $\Psi(F_{m})\,=\, F_{m}$. Note also
\begin{equation*}
F_{i,j}\,\cap\, F_{i,j^{\prime}}\,=\,\emptyset \,\quad\textrm{for}\quad j\,\neq\,j^{\prime}, \,\,\textrm{and}\,\,F_{i,j}\,\cap\, F_{i^{\prime}, j^{\prime}}\,\neq\, \emptyset\,\, \textrm{whenever}\,\, i\,\neq\, i^{\prime}.
\end{equation*}

Hence, there exists a permutation $\sigma\,\in\, S_{m}$ such that for all $i$, we have $\{\Psi(F_{i,j})\}_{j}\,=\,\{F_{\sigma(i),j}\}_{j}$.

So, by precomposing $\Phi$ and $\Psi$ with some element of $S_m$, we may assume
\begin{equation*}
\{\Psi(F_{i,j})\}_{j}\,=\, \{F_{i,j}\}_{j}\quad \textrm{for each}\,\, i.
\end{equation*}
Let $\pi_i\,:\, Y^{m}\,\longrightarrow\, Y$ be the projections, and $\Psi_{i}\,=\, \pi_{i}\,\circ\, \Psi$  be the components of $\Psi$. We have $\Psi_{i}(\pi^{-1}_{i}(z_{j}))$ is a point, so by rigidity lemma (\cite[Page 40]{Mum}) $\Psi_{i}$ factors through $\pi_{i}$. So, there are $\psi_{i}: Y\longrightarrow\, Y$ for $1\,\leq i\,\leq m$ such that $\Psi(\underline{y})\,=\,(\psi_{1}(y_{1}),\ldots, \psi_{m}(y_{m}))$ for all $\underline{y}\in Y^m$. As $\Psi|_{\Delta_Y}\,=\,$ id, we see that each $\psi_i$ is the identity map. So, $\Psi\,=\, id$. Hence $\Phi\,=\, id$ too.
\end{proof}

Before proving the other two theorems, we need a few lemmas.

\begin{lemma}\label{open}
Let $C, F$ be smooth curves, $F$ (but not necessarily $C$) projective. Let $X\,=\, F\,\times\, C$, $X\,\xlongrightarrow{p}\,C$ the projection. Let $m\,\geq\, 2$, $\phi$ an $S_m$-equivariant automorphism of $X^{m}$ over $C^{m}$, such that $\phi{|_{\Delta_{X}}}\,=\, id$. For $1\,\leq\,i\,\leq\, m$, let $W_{i}\,=\, C^{i-1}\,\times\, X\,\times\, C^{m-i}$, $\pi=p_{(i-1)}\times id \times p_{(m-i)}: X\,\to \, W_{i}$. Then either $\phi$ descends to an automorphism of $W_{i}$ for all $i$, or $m\,=\,2$ and $\phi_{1}$ factors through $X\,\times\, X\,\xlongrightarrow{p\times id}\, C\,\times\,X$, $\phi_{2}$ factors through $X\,\times\, X\,\xlongrightarrow{id\times p}\, X\,\times\, C$.
\end{lemma}

\begin{proof}
We have the commutative diagram:

\begin{center}
\begin{tikzcd}
 (F\,\times\, C)^{m}  \arrow[rr, "\phi"] \arrow[swap]{dr}{p_{(m)}} &  _{}  & (F\,\times\,C)^{m}\arrow{dl}{p_{(m)}} \\[5pt]
 & C^{m}
\end{tikzcd}.
\end{center}
So, we get a morphism 
\begin{equation}\label{map to aut}
C^{m}\,\xlongrightarrow{T}\, \textrm{Aut}\,F^{m}.
\end{equation}

We will write elements of $X$ as $(z,c)$, with $z\,\in\, F$, $c\,\in\, C.$
We deal with the two cases separately.

\underline{Case\,1}: $F\,=\, \mathbb{P}^{1}$ or $g(F)\,\geq\, 2$.

By Lemma \ref{autoproduct}, $\textrm{Aut}(F^{m})\,=\, (\textrm{Aut}F)^{m} \rtimes\, S_{m}$. As, $S_{m}$ is discrete, there is $\sigma\,\in\, S_{m}$, $\tau_{i,\underline{c}}\,\in\, \textrm{Aut}(F)$ for $\underline{c}\,\in\, C^{m}$, $1\,\leq\,i\,\leq\,m$ such that
\begin{equation*}
T(\underline{c})(\underline{z})\,=\, (\tau_{i,\underline{c}}(z_{\sigma(i)}))_{i},    
\end{equation*} for all $\underline{z}\,\in\, F^m.$
So, 
\begin{equation}\label{horr p1}
\phi((z_{i},c_{i})_{i})\,=\, (\tau_{i,\underline{c}}(z_{\sigma(i)}), c_{i})_{i}. 
\end{equation}
If $\sigma\,=\, 1$, then $\phi$ descends to $W_{1}\,\xlongrightarrow{\psi_{1}}\,W_{1}$ given by 

\begin{equation*}
\psi_{1}((z_{1},c_{1}), c_{2},\ldots, c_{m}) \,=\, \left((\tau_{1,\underline{c}}(z_{1}), c_{1}),c_{2},\ldots,c_{m}\right)   
\end{equation*}
for $z_{1}\,\in\,F$, $\underline{c}\,\in\,C^{m}$. Similarly, $\phi$ descends to each $W_{i}$. 

So, assume $\sigma\,\neq\,1$. As $\phi$ is $S_m$-equivariant, for every $\pi\,\in\,S_{m}$, we have $\phi(\underline{x}_{\pi})\,=\, \phi(\underline{x})_{\pi}$. By \eqref{horr p1}, if $x_{i}\,=\, (z_{i},c_{i})\,\in\,F\,\times\, C,$
\begin{equation*}
\phi(\underline{x}_{\pi})\,=\, (\tau_{i, \underline{c}_{\pi}}(z_{\pi\,\sigma(i)}), c_{\pi(i)})_{i}     
\end{equation*}
and
\begin{equation*}
\phi(\underline{x})_{\pi}\,=\, (\tau_{\pi(i),\underline{c}}(z_{\sigma \pi{(i)}}), c_{\pi(i)})_i  
\end{equation*}
for all $i,\underline{c}, \underline{z}$. So, 
\begin{equation}\label{terr p1}
\tau_{i, \underline{c}_{\pi}}(z_{\pi\sigma{(i)}})\,=\, \tau_{\pi(i),\underline{c}}(z_{\sigma\pi(i)})  
\end{equation}
for all $\pi\,\in\, S_{m}$, $\underline{c}\,\in\, C^{m}, \underline{z}\,\in\,F^{m}$, $1\,\leq\, i\,\leq\, m$. If $m\,\geq\,3$, then since $\mathcal{Z}(S_{m})\,=\, 1$ and $\sigma\,\neq\, 1$, we can choose $\pi\,\in\, S_{m}$ such that $\pi\,\sigma\,\neq\, \sigma\,\pi$. Choose $i$ such that $\pi\,\sigma(i)\,\neq\,\sigma\,\pi(i)$. As $\eqref{terr p1}$ is true for all $\underline{z}$, we see that $\tau_{i,\underline{c}_{\pi}}$, $\tau_{\pi(i),\underline{c}}$ has to be constant function. But they are in $\textrm{Aut}\,F$, a contraction. 

So, $m\,=\,2$ and $\sigma$ is the transposition. Define

\begin{equation*}
\psi_{1}\,:\, C\,\times \,X\,\longrightarrow\, X    
\end{equation*}
\begin{equation*}
(c_{1},(z_{2}, c_{2}))\,\longmapsto\, (\tau_{1,\underline{c}}(z_{2}),c_{1})    
\end{equation*}
\begin{equation*}
\psi_{2}\,:\, X\,\times\,C\,\longrightarrow\, X    
\end{equation*}
\begin{equation*}
((z_{1},c_{1}), c_{2})\,\longmapsto\, (\tau_{2,\underline{c}}(z_{1}), c_{2}).    
\end{equation*}
Using $\eqref{horr p1}$, one sees that $\phi_{i}$ factors through $\psi_{i}$.

\underline{Case\,2}: $F\,=\, E$, an elliptic curve.

We have $\textrm{Aut}(E^{m})\,=\, E^{m}\,\rtimes\,\textrm{End}E^{m}$. As $\textrm{End}(E^{m})$ is discrete, there is $P\,\in\, \textrm{End}(E^{m})$ and  $\underline{a}_{\underline{c}}\,\in\, E^{m}$ for $\underline{c}\,\in\, C^{m}$ such that 
\begin{equation*}
T(\underline{c})(\underline{z})\,=\, P(\underline{z})\,+\, \underline{a}_{\underline{c}}.  
\end{equation*}
Let $P_{i}\,:\,E^{m}\,\longrightarrow\, E$ be the components of $P$, $a_{\underline{c},i}\,\in\, E$ the components of $\underline{a}_{\underline{c}},$ $1\,\leq\,i\,\leq\,m$. So, 
\begin{equation}\label{horr E}
\phi((z_{i},c_{i})_{i})\,=\, (P_{i}(\underline{z})+a_{\underline{c},i},c_{i})_{i}.    
\end{equation}
If $P\,=\,{id}$, then $\phi$ descends to $W_{1}\,\xlongrightarrow{\psi_{1}}\,W_{1}$ given by $\psi_{1}((z_{1},c_{1}), c_{2},\ldots,c_{m})\,=\, ((z_{1}+a_{\underline{c},1},c_{1}),c_{2},\ldots,c_{m})$, for $z_{1}\,\in\, E, \underline{c}\,\in\, C^{m}$. Similarly, $\phi$ descends to each $W_{i}$. 

So, assume $P\,\neq\,{id}$. As $\phi$ is $S_m$-equivariant, for every $\sigma\,\in\,S_{m}$, we have 
\begin{equation*}
\phi(\underline{x}_{\sigma})\,=\, \phi(\underline{x})_{\sigma} \quad \textrm{for all}\,\, \underline{x}\,\in\, X^{m}.
\end{equation*}
By \eqref{horr E}, if $x_{i}\,=\, (z_{i},c_{i})\,\in\, E\,\times\,C$,
\begin{equation*}
\phi(\underline{x}_{\sigma})\,=\, (P_{i}(\underline{z}_{\sigma})\,+\, a_{\underline{c}_{\sigma},i}, c_{\sigma(i)})_{i},    
\end{equation*}
and
\begin{equation*}
\phi(\underline{x})_{\sigma}\,=\, (P_{\sigma(i)}(\underline{z})\,+\, a_{\underline{c},\sigma(i)}, c_{\sigma(i)} )_{i}.   
\end{equation*}
 So, we have
\begin{equation*}
P_{i}(\underline{z}_{\sigma})\,+\, a_{\underline{c}_{\sigma},i}\,=\, P_{\sigma(i)}(\underline{z})\,+\, a_{\underline{c}, \sigma(i)} 
\end{equation*}
for all $\sigma\,\in\, S_{m}$, $\underline{c}\,\in\, C^{m},$ $ \underline{z}\,\in\, E^{m}$, $1\,\leq\,i\,\leq\, m$. This forces
$P_{i}(\underline{z}_{\sigma})\,=\, P_{\sigma(i)}(\underline{z})$. Hence, $P\,:\, E^{m}\,\longrightarrow\, E^{m}$ is $S_{m}$-equivariant, and $P|_{\Delta_{E}}\,=\,{id}$ as $\phi{|_{\Delta_{E}}}\,=\,{id}$. By Corollary \ref{nounit}, there is $\tau\,\in\, S_{m}$ such that $P(\underline{z})\,=\, \underline{z}_{\tau}$.

As $P$ is $S_{m}$--equivariant, we get $\tau\,\in\, \mathcal{Z}(S_{m})$. As $P\,\neq\, {id}$, we get $\mathcal{Z}(S_{m})\neq 1$. So $m\,=\,2$ and $\tau$ is the transposition. Define
\begin{equation*}
\psi_{1}\,:\, C\,\times\, X\,\longrightarrow\, X    
\end{equation*}
\begin{equation*}
(c_{1}, (z_{2}, c_{2}))\,\mapsto \, (z_{2}\,+\, a_{\underline{c},1},c_{1})    
\end{equation*}
and 
\begin{equation*}
\psi_{2}\,:\, X\,\times\, C\,\longrightarrow\, X    
\end{equation*}
\begin{equation*}
((z_{1}, c_{1}), c_{2})\,\mapsto\,(z_{1}\,+\, a_{\underline{c},2}, c_{2}).    
\end{equation*}
Using $\eqref{horr E},$ one sees that $\phi_{i}$ factors through $\psi_{i}$.
\end{proof}
\begin{lemma}\label{family}
Let $Z$ be a finite type $\mathbb{C}$-scheme, $B$ a variety, $Z\,\xlongrightarrow{f}\,B$ a flat proper map. Suppose $L$ is a globally generated line bundle on $Z$, $B_{0}\,\subset\, B$ a dense open set. If $\left. L \right|_{f^{-1}(b)}$ is trivial for all $b\,\in\, B_{0}$, then the same is true for all $b\,\in\,B$.
\end{lemma}

\begin{proof}
We can assume $B$ is affine. Since any two points of $B$ lie in an irreducible curve by {\cite[Lemma, Page 53]{mumford1974abelian}}, we can assume $B$ is an irreducible affine curve . Since $L$ is globally generated, we get a morphism $Z\,\xlongrightarrow\,\mathbb{P}^{n}_{B}$ over $B$ such that $\phi^{\ast}\mathcal{O}(1)\,=\, L$. Let $W\,=\, \phi(Z)$, the scheme-theoretic image of $\phi$. So, we have a commutative diagram:

\begin{center}
\begin{tikzcd}
 Z  \arrow[rr, "\phi"] \arrow[swap]{dr}{f} &  _{}  & W\arrow{dl}{g} \\[5pt]
 & B
\end{tikzcd}.
\end{center}
Since $f$ is flat, every irreducible component of $Z$ dominates $B$. So the same is true for $W$. As $\left. L \right|_{f^{-1}(b)}$ is trivial for all $b\,\in\,B_{0}$, we see that $g^{-1}(b)\,=\,\phi(f^{-1}(b))$ is a finite set for all $b\,\in\, B_{0}$. So, $g^{-1}(B_{0})\to  B_{0}$ is finite. As every irreducible component of $W$ dominate $B$ and $B$ is a curve, we see that $g$ is finite. So, $\left. L \right|_{f^{-1}(b)}$ is trivial for all $b\,\in \, B$.
\end{proof}

\begin{corollary}\label{autodescent}
Let 
\begin{equation*}
\begin{tikzcd} 
Z^{\prime} \arrow[r, "h^{\prime}"   ] \arrow[d,""]
&  Z\arrow[d , swap,"f"] \\
Y^{\prime}\arrow[r, "h"]
& |[, rotate=0]|  Y
\end{tikzcd}
\end{equation*}
be a fibre square of finite type $\mathbb{C}$-schemes, $Z, Y, Y^{\prime}$ varieties, $Z, Y$ projective, $f$ a flat contraction and $h$ is dominant. Let $\phi\,\in\,\emph{Aut}\,Z$, $\phi^{\prime}\,\in\, \emph{Aut}\,Z^{\prime}$, and $h^{\prime}\,\phi^{\prime}\,=\, \phi\,h^{\prime}$. Suppose $\phi^{\prime}$ descends to an automorphism of $Y^{\prime}$. Then $\phi$ descends to an automorphism of $Y$.
\end{corollary}
\begin{proof}
Let $L$ be a very ample line bundle on $Y$. It suffices to show that $f\,\phi(C)$ is a point whenever $f(C)$ is a point for a curve $C$ in $Z$. So, it is enough to show that $M\,:\,= \phi^{\ast}f^{\ast}L$ is trivial on fibres of $f$. Since $M$ is globally generated and is trivial on fibres of $f$ over points in $h(Y^{\prime})$, we are done by Lemma \ref{family}.
\end{proof}

\begin{corollary}\label{autodescent2}
Let the solid arrows of
\begin{center}
\begin{tikzcd}
Z^{\prime} \arrow[d,""'] \arrow[r,""] &
Z\arrow[d,"f"'] \arrow[ddr,bend left,"g"] \\
Y^{\prime} \arrow[r,"h"] \arrow[drr,bend right,""'] &
Y \arrow[dr,dashed,""] \\
&& W
\end{tikzcd}
\end{center}
form a commutative diagram of finite type $\mathbb{C}$--schemes, $Z, Y, Y^{\prime}, W$ varieties, $Z, Y, W$ projective, $f$ a flat contraction, $h$ is dominant and the square is a fibre square. Then the dotted arrow exists, making the whole diagram commute.
\end{corollary}
\begin{proof}
Let $L$ be a very ample line bundle on $W$. It suffices to show that $g(C)$ is a point whenever $f(C)$ is a point for a curve $C$ in $Z$. So, it is enough to show that $M\,:\,=\, g^{\ast}L$ is trivial on fibres of $f$. Since $M$ is globally generated and is trivial on fibres of $f$ over points in $h(Y^{\prime})$, we are done by Lemma \ref{family}.
\end{proof}

Now we are ready to prove Theorems \ref{isotrivial} and \ref{non-isotrivial}.

\textit{Proof of Theorem \ref{isotrivial}:}
\begin{proof}
Let $m\,\geq\, 2$, and let $X^{m}\,\xlongrightarrow{p_{(m)}}\,C^{m}$ be the map induced by $p$. For $c\,\in \, C$, let $X_{c}\,=\, p^{-1}(c)$ be the fiber over $c$. Let $\phi^{\prime}$ be an automorphism of $S^{m}X$.  Our goal is to show that $\phi^{\prime}\,=\, {id}$.
By Lemma \ref{equivariant}, $\phi^{\prime}$ lifts to an $S_m$--equivariant automorphism $\phi\,:\, X^{m}\,\longrightarrow\, X^{m}$. By composing $\phi^{\prime}$ with a natural automorphism if necessary, we may assume $\phi|_{\Delta_X}\,=\, {id}$. Let $\phi_{i}\,:\, X^{m}\,\longrightarrow\, X$ be the components of $\phi$. We now divide the proof into several steps.

\underline{Step\,1}: We claim that $\phi$ descends to an automorphism of $C^{m}$ under $p_{(m)}$. 

If $(i)$ holds, then the general fiber of $p_{(m)}$ is contracted by $p_{m}\,\circ\, \phi$, so $p_{(m)}$ descends to $C^m$ by Corollary \ref{autodescent2}. If $(ii)$ holds, then $p_{(m)}$ is the contraction part of the Albanese map of $X^m$, so $\phi$ descends to $C^{m}$. If $(iii)$ holds, then $p_{(m)}$ is the Iitaka fibration of $X^m$, so $\phi$ descends to $C^{m}$. 

Applying this to $\phi^{-1}$ as well, we conclude that $\phi$ descends to an $S_m$--equivariant automorphism $\psi$ of $C^{m}$. Since $\phi|_{{\Delta_{X}}}\,=\, {id}$, we have $\psi|_{{\Delta_{C}}}\,=\, {id}$.

\underline{Step\,2}: We show that, possibly after precomposing $\phi$ and $\psi$ with an element of $\mathcal{Z}(S_m)$, we have $\psi\,=\, {id}$, i.e., $\phi$ is over $C^{m}$. 

We consider 2 cases separately.

\underline{Case\,1}: $C\,=\,\mathbb{P}^1$ or $g(C)\,\geq\,2$

By Lemma $\ref{autoproduct}$,  there are $\tau_{i}\,\in\, \textrm{Aut}(C)$ for $1\,\leq\,i\,\leq\,m$ and $\sigma\in S_m$ such that $\psi(\underline{c})\,=\, (\tau_{1}(c_{\sigma(1)}),\ldots,\, \tau_{m}(c_{\sigma(m)}))$ for all $\underline{c}\,\in\, C^{m}$. Since $\psi$ is $S_m$--equivariant and $\psi_{|_{\Delta_{c}}}\,=\,id$, we get all $\tau_i$'s are $id$ and $\sigma\,\in\, \mathcal{Z}(S_m)$. 

\underline{Case\,2}: $g(C)\,=\, 1$. 

Since $C$ is an elliptic curve, we are done by Corollary \ref{nounit}.   

\underline{Step\,3}: We now have the commutative diagram:

\begin{center}
\begin{tikzcd}
 X^{m}  \arrow[rr, "\phi"] \arrow[swap]{dr}{p_{(m)}} &  _{}  & X^{m}\arrow{dl}{p_{(m)}} \\[5pt]
 & C^{m}
\end{tikzcd},
\end{center}
which implies:
\begin{equation}\label{commutativity}
p\,\phi_{i}(\underline{x})\,=\, p(x_{i})    
\end{equation}
for all $i$ and $\underline{x}\,\in\, X^{m}$. For $1\,\leq\,i\,\leq m$, let
$$Z_{i}\,=\, C^{i-1}\,\times\,X\,\times \, C^{m-i} \quad \text{and}\quad X^{m}\,\xlongrightarrow{\pi_i}\, Z_{i}$$ over $C^{m}$ given by $\pi_{i}=p_{(i-1)}\times id \times p_{(m-i)}$. The morphism $\pi_{i}$ is a flat contraction as $\pi_{i}$ is a product of flat contractions $p$ and $id$. In this step, we prove the following claim.  

\noindent\underline{Claim}:  $\phi$ descends to an automorphism of $Z_{i}$ for all $i$.

\noindent\underline{Proof}: Since $p$ is isotrivial with general fiber $F$, there exists a smooth curve $D$, not necessarily proper, and a dominant map $D\,\xlongrightarrow{h}\,C$ such that $X_{D}\,:=\, X\,\times_{C}\,D\,\cong\,F\,\times\, D$ over $D$. Let $p_D$ be the base change of $p$ under $h$, $Z_{i, D}$, $\phi_{D}$ be the base change of $Z_{i}$ and $\phi$ to $D^{m}$ respectively, under the map $D^{m}\,\xlongrightarrow{h_{(m)}}\, C^{m}$ induced by $h$, and $\phi_{i,D}$ the components of $\phi_D$ for $1\,\leq\, i\,\leq\, m$. Note that $\phi_{D}$ is $S_m$--equivariant and $\phi_{D}|_{\Delta_{X_{D}}}\,=\, {id}$. Applying Lemma $\ref{open}$ applied to $X_{D}$, two possibilities can arise:

\underline{Case 1}: $\phi_{D}$ descends to an automorphism of $Z_{i,D}$, in which case the claim follows from Lemma \ref{autodescent}.

\underline{Case 2}: $m\,=\, 2$, there exists maps $D\,\times\, X_{D}\,\xlongrightarrow{\psi_{1,D}}\, X_{D}$, and $X_{D}\,\times\,D\,\xlongrightarrow{\psi_{{2},D}}\,X_{D}$ such that 
\begin{equation*}
\phi_{1,D}(x_{1}, x_{2})\,=\, \psi_{1,D}(p(x_{1}), x_{2}),    
\end{equation*}
\begin{equation*}
\phi_{2,D}(x_{1},x_{2})\,=\, \psi_{2,D}(x_{1}, p(x_{2})),    
\end{equation*}
for all $x_{1},x_{2}\,\in\,X_{D}$.
We have a commutative diagram:

\begin{center}
\begin{tikzcd}
X_{D}\,\times\, X_{D} \arrow[d,"p_{D}\times{id}"'] \arrow[r,""] &
X\,\times\,X \arrow[d,"p\times{id}"'] \arrow[ddr, "\phi_{1}"] \\
D\,\times\, X_{D} \arrow[r,""]  \arrow[dr,"\psi_{1,D}"'] &
C\,\times\, X  \\
& X_{D} \arrow[r,""'] 
& X
\end{tikzcd},
\end{center}
where the square is a fibre square. By Lemma \ref{autodescent2}, $\phi_{1}$ factors through a map $\psi_{1}\,:\, C\,\times\, X\,\longrightarrow\, X$. Similarly, $\phi_{2}$ factors through a map $\psi_{2}\,:\, X\,\times\, C\,\longrightarrow\, X$. In other words, 
\begin{equation}\label{phipsi}
\phi_{1}(x_{1}, x_{2})\,=\, \psi_{1}(p(x_{1}),x_{2}) \quad\textrm{and}\quad  \phi_{2}(x_{1},x_{2})\,=\, \psi_{2}(x_{1}, p(x_{2}))
\end{equation}
for all $x_{1}, x_{2}\,\in\, X$. By \eqref{commutativity}, we have:
\begin{equation}\label{commutativity2}
p\, \psi_{1}(c,x)\,=\, c\,=\, p\, \psi_{2} (x,c)    
\end{equation}
for all $c\,\in \,C$ and $x\,\in\, X$.

As $\phi$ is $S_2$--equivariant, we have:
\begin{equation*}
\phi_{1}(x_{2},x_{1})\,=\, \phi_{2}(x_{1}, x_{2}),   
\end{equation*}
for all $x_{1}, x_{2}\,\in\, X$. By $\eqref{phipsi}$, $\psi_{1}(p(x_{1}),x_{2})\,=\, \psi_{2}(x_{2},p(x_{1}))$ for all $x_{1}, x_{2}\,\in\, X$. This means 
\begin{equation*}
\psi_{1}(c,x)\,=\, \psi_{2}(x,c)    
\end{equation*}
for all $c\,\in\, C$ and $x\,\in\, X$. For $c\,\in\,C$, let $\psi_{c}\,:\, X\,\longrightarrow\, X$
be given by $\psi_{c}(x)\,=\,\psi_{1}(c, x)\,=\, \psi_{2}(x,c)$. By \eqref{commutativity2}, $p\,\psi_{c}(x)\,=\,c$, so, $\psi_{c}(X)\,\subset\, X_{c}$. 

By \eqref{phipsi}, we have:
\begin{equation*}
\phi(x_{1},x_{2})\,=\, (\psi_{p(x_{1})}(x_{2}), \psi_{p(x_{2})}(x_{1}))    
\end{equation*}
for all $x_{1}, x_{2}\,\in\, X$. For $c_{1},c_{2}\,\in\,C$, the restriction of $\phi$ to $X_{c_{1}}\,\times\, X_{c_{2}}$ is an isomorphism and it is given by 
\begin{equation*}
\phi(x_{1},x_{2})\,=\, (\psi_{c_{1}}(x_{2}), \psi_{c_{2}}(x_{1})).    
\end{equation*}
It follows that $\psi_{c_{1}}\,:\, X_{c_{2}}\,\longrightarrow\,X_{c_{1}}$ is an isomorphism, for any $c_{1}, c_{2}\,\in\, C$.

Now define
\begin{equation*}
X\,\xlongrightarrow{\Phi}\, C\,\times\, F
\end{equation*}
\begin{equation*}
x\,\longmapsto\, (p(x), \psi_{c_{0}}(x))    
\end{equation*}
Choose general $c_{0}\,\in\, C$, identify $X_{c_{0}}\,=\,F$. Note that $\Phi$ is a morphism over $C$, and over each point of $C$, $\Phi$ is an isomorphism. So, $\Phi$ is an isomorphism. So, $(1)$ of Theorem \ref{4} holds, a contradiction. This finishes the proof of the claim.

\underline{Step 4}: We now complete the proof.

Let $\psi_{1}\,:\, Z_{1}\,\longrightarrow\, Z_{1}$ be the descent of $\phi$, $\psi_{11}\,:\, Z_{1}\,\longrightarrow\,X$ the first component of $\psi_{1}$. So,
\begin{equation}\label{trans}
\phi_{1}(\underline{x})\,=\, \psi_{11}(x_{1}, p(x_{2}),\ldots, p(x_{m})).
\end{equation}
Also, as $\phi{|_{\Delta_{X}}}\,=\, {id}$, we have
\begin{equation}\label{diagonal}
\psi_{11}(x, p(x),\ldots,p(x))\,=\, x    
\end{equation}
for all $x\,\in\, X$. By \eqref{commutativity}, and \eqref{trans}, $p\,\psi_{11}(x_{1}, c_{2},\ldots, c_{m})\,=\, p(x_{1})$. 

So, we can define a morphism $C^{m-1}\,\xlongrightarrow{T}\, \textrm{Aut}_{C}X$ given by
\begin{equation*}
T(c_{2},\ldots\, c_{m})(x_{1})\,=\, \psi_{11}(x_{1},c_{2},\ldots, c_{m}).    
\end{equation*}
As $(4)$ of Theorem \ref{4} does not hold, by Lemma \ref{equivalence}, $T$ is constant. By \eqref{diagonal}, $T(\underline{c})\,=\, {id}$ for all $\underline{c}$. So, $\psi_{11}(x_{1},c_{2},\ldots,c_{m})\,=\,x_{1}$. So, $\phi_{1}(\underline{x})\,=\, x_{1}$ for all $\underline{x}$.

Similarly, $\phi_{i}(\underline{x})\,=\, x_{i}$ for all $\underline{x}$. So, $\phi\,=\, id$.
\end{proof}

\textit{Proof of Theorem \ref{non-isotrivial}:}
\begin{proof}
Let $X^{m}\,\xlongrightarrow{p_{(m)}}\, C^{m}$ be the map induced from $p$. For $c\,\in \, C$, let $X_{c}\,=\, p^{-1}(c)$. Given an automorphism of $S^{m}X$, we lift it to an $S_m$--equivariant automorphism $X^{m}\,\xlongrightarrow{\phi}X^{m}$ by Lemma \ref{equivariant}.  Similar to Steps $1$ and $2$ of Theorem \ref{isotrivial}, we can asssume $\phi$ is over $C^{m}$ and $\phi_{|_{\Delta_{X}}}\,=\,{id}$. Let $\phi_{i}\,:\, X^{m}\,\longrightarrow\, X$ be the components of $\phi$. As in $\eqref{commutativity}$, we have 
\begin{equation*}
p\,\phi_{i}(\underline{x})\,=\, p(x_{i}).    
\end{equation*}
As $p$ is not isotrivial, there are uncountably many distinct isomorphism classes of fibres of $p$. So, by \cite[Theorem 2]{Mae}, for a very general $x_{1}\,\in\, X$, there is no non-constant map $X^{m-1}\,\longrightarrow\, X_{p(x_{1})}$. So, the map 
$\phi_1 \big|_{\{x_1\}\, \times\, X^{m-1}}
\,:\,X^{m-1}\,\longrightarrow\, X_{p(x_{1})}\,\longrightarrow\,X$ must be constant. By rigidity lemma (\cite[Page 40]{Mum}), there is $\overline{\phi_{1}}\,:\, X\,\longrightarrow\, X$ such that $\phi_{1}(\underline{x})\,=\,\overline{\phi_{1}}(x_{1})$ for all $\underline{x}$. As $\phi{|_{\Delta_{X}}}\,=\, {id}$, we see that $\overline{\phi_{1}}\,=\, {id}$. So, $\phi_{1}(\underline{x})\,=\, x_{1}$, for all $\underline{x}$. Similarly, $\phi_{i}(\underline{x})\,=\,x_{i}$ for all $\underline{x}$. So, $\phi\,=\, id$.
\end{proof}

\textit{Proof of Corollary \ref{c1}}:
If $(a)$ or $(d)$ of Theorem \ref{1} holds, then $\textrm{Hilb}^m(X)$ admits a non-natural automorphism by Theorem \ref{1}. So suppose that neither $(a)$ nor $(d)$ holds. We aim to show that every automorphism $\phi$ of $\textrm{Hilb}^m(X)$ is natural. If $\kappa(X)\,=\,2$, this follows from \cite{belmans2020automorphisms}. So we assume  $\kappa(X)\,=\,1$. In that case, $\phi$ preserves the big diagonal by the same argument as in the proof of {\cite[Theorem 3.3]{Ha}}, noting that the movable part of $NK_X$ is semi-ample for sufficiently large $N$, by the abundance theorem for surfaces. Hence, we are done by Theorem \ref{1}.

\begin{remark}\label{determine1}
By the same argument as in the proof of Corollary \ref{c1} or {\cite[Theorem 3.3]{Ha}}, for smooth projective surfaces $X$ and $Y$ of Kodaira dimension $\geq\, 1$, any isomorphism $\textrm{Hilb}^{m}X\,\to\, \textrm{Hilb}^{m}Y$ maps the big diagonal to the big diagonal. Therefore, $X\,\cong\, Y$ by Remark \ref{determine}. In particular, $X$ is determined by $\textrm{Hilb}^{m}X$.
\end{remark}
\begin{remark}
    In fact, the proofs of Theorems \ref{nonminimal} and \ref{non-isotrivial} shows a somewhat stronger statement: if $X$ is as in Theorem \ref{nonminimal} or as in Theorem \ref{non-isotrivial} and $C$ is not an elliptic curve, then any automorphism of $X^m$ preserving $\Delta_X$ is a composition of a natural automorphism and a permutation.
\end{remark}
\textit{Proof of Theorem \ref{3} $(ii):$}

\begin{proof}
This follows essentially the same strategy as the proof of Theorem \ref{isotrivial}. We follow the same proof and notations with the $\mathbb{P}^r$--bundle $X\,\xrightarrow{p}\,Y$ in place of $X\,\xrightarrow{p}\,C$. Since hypothesis $(i)$ of Theorem \ref{isotrivial} holds, Steps $1$ and $2$ works as it is. Step $4$ also works: since $\textrm{Aut}_{Y}X$ has no complete subvariety, any morphism $Y^{m-1}\,\xlongrightarrow{T}\, \textrm{Aut}_{Y}X$ must be constant. It remains to verify the claim in Step $3$.

If $X\,\cong\,\mathbb{P}^r\,\times\, Y $ over $Y$, then $\phi$ induces a morphism $Y^m\,\to \,\text{Aut}(\mathbb{P}^r)^m$, which must be constant as $\text{Aut}(\mathbb{P}^r)^m$ is affine. Now we get $m\,=\,2$ as in Case 1 in the proof of Lemma \ref{open}, contradicting our assumption.

Let $q_i\,:\,Y^m\,\to\, Y$ denote the projections, and let $\mathcal{E}$ be a vector bundle on $Y$ such that $X\,=\,\mathbb{P}(\mathcal{E}).$ Let $\mathcal{E}_i\,=\,q_i^*\mathcal{E}.$ Since $\mathcal{E}$ is not trivial up to line bundle twist by the argument above, the projective bundles $\mathbb{P}(\mathcal{E}_i)$ are pairwise non-isomorphic over $Y^m$. The claim then follows from {\cite[Lemma 2.6]{BSV4}}, noting that $X^m$ is the multiprojective bundle $\mathbb{P}(\mathcal{E}_1,\ldots,\mathcal{E}_m)$ over $Y^m.$
\end{proof}
\begin{remark}
For a smooth projective variety $X$ of dimension $\geq\, 2$ and $m\,\geq\, 2$. Let $G_m(X)$ be the kernel of the natural surjection $\textnormal{Aut} (S^mX)\,\longrightarrow\, \textnormal{Aut}(X)$. As in Theorem \ref{3}, $\textnormal{Aut} (S^mX)$ is a semidirect product of $G_m(X)$ and $\textnormal{Aut} (X).$ By {\cite[Remark 3.1]{BSV3}}, we have $\textnormal{Aut}^{0} (S^mX)\,=\,\textnormal{Aut}^{0} (X)$, hence $G_m(X)$ is a discrete group. It is interesting to study this group for higher dimensional varieties $X.$ For example, by an argument similar to Lemma \ref{unit}, one can show that for an abelian variety $X$ of dimension $\geq\, 2$, the group $G_m(X)$ coincides with the congruence subgroup 
$$\{\alpha \,\in\, \textrm{End}(X)^{\times}\,|\, m \text{ divides } 1\,-\,\alpha \text{ in }\textrm{End}(X)\}. $$ 
In the case of smooth projective surfaces, it is not difficult to show that the examples of non-natural automorphisms given in parts $(1)-(4)$ of Theorem \ref{4} are essentially the only possibilities.
\end{remark}

\section{Acknowledgements}\vskip 3mm We thank János Kollár, Joaquin Moraga, Bhargav Bhatt, the authors of \cite{belmans2020automorphisms}, Daniel Erman, Michele Graffeo, and Anthony Iarrobino for insightful discussions.

\printbibliography
\end{document}